\documentclass[11pt,openany,leqno, english]{article}
\usepackage{yhmath, amsmath,amsthm,amsfonts,amssymb,amscd,url}
\usepackage{graphics}
\usepackage[latin1]{inputenc}
\usepackage{enumerate}
\begin{document}
\def\RA{{\bf\text{To be read again}}}
\def\ERA{{\bf\text{End To be read again}}}
\def\Diff{\text{Diff}}
\def\Max{\text{max}}
\def\Log{\text{log}}
\def\loc{\text{loc}}
\def\inta{\text{int }}
\def\det{\text{det}}
\def\exp{\text{exp}}
\def\Re{\text{Re}}
\def\lip{\text{Lip}}
\def\leb{\text{Leb}}
\def\dom{\text{Dom}}
\def\diam{\text{diam}\:}
\newcommand{\ovfork}{{\overline{\pitchfork}}}
\newcommand{\ovforki}{{\overline{\pitchfork}_{I}}}
\newcommand{\Tfork}{{\cap\!\!\!\!^\mathrm{T}}}
\newcommand{\whforki}{{\widehat{\pitchfork}_{I}}}
\theoremstyle{plain}
\newtheorem{theo}{\bf Theorem}[section]
\newtheorem{lemm}[theo]{\bf Lemma}
\newtheorem{prop}[theo]{\bf Proposition}
\newtheorem{coro}[theo]{\bf Corollary}
\newtheorem{Property}[theo]{\bf Property}
\newtheorem{conj}[theo]{\bf Conjecture}

\theoremstyle{remark}
\newtheorem{rema}[theo]{\bf Remark}
\newtheorem{exem}[theo]{\bf Example}
\newtheorem{Examples}[theo]{\bf Examples}
\newtheorem{defi}[theo]{\bf Definition}
\newtheorem{ques}[theo]{\bf Question}
\newtheorem{claim}[theo]{\bf Claim}
\newtheorem{fact}[theo]{\bf Fact}
\newtheorem{sublemma}[theo]{\bf Sublemma}

\title{Persistent bundles over a 2 dimensional compact set}
\date{}
\author{Pierre Berger\\
 CNRS-LAGA Université Paris 13\\
berger@math.univ-paris13.fr}

\maketitle

\begin{abstract} 
The $C^1$-structurally stable diffeomorphims of a compact manifold are 
those that satisfy Axiom $A$ and the strong transversality condition ($AS$). We generalize the concept of $AS$ from diffeomorphisms to invariant compact subsets. Among other properties, we show the structural stability of the $AS$ invariant compact sets $K$ of surface diffeomorphisms $f$. Moreover if $\hat f$ is the dynamics of a compact manifold, which fibers over $f$ and such that the bundle is normally hyperbolic over the nonwandering set of $f{|K}$, then the bundle over $K$ is persistent. This provides non trivial examples of persistent laminations that are not normally hyperbolic.
\end{abstract}


\section*{}

 A classical result states that
hyperbolic compact sets are $C^1$-structurally stable.
A compact subset $K$ of a manifold $M$ is \emph{hyperbolic} for a diffeomorphism $f$ of $M$ if it is invariant ($f(K)=K$) and the tangent bundle of $M$
restricted to $K$ splits into one or two $Tf$-invariant subbundles
 contracted or expanded.
An invariant subset $K$ of a diffeomorphism $f$ is \emph{$C^1$-structurally stable} if every
$C^1$-perturbation $f'$ of $f$ leaves invariant a compact set $K'$
homeomorphic to $K$ by an embedding $C^0$-close to the inclusion $K
\hookrightarrow M$ which conjugates the dynamics $f{|K}$ and $f'{|K'}$.

Such a result was generalized toward two directions that we would love to unify.

The first was to describe the \emph{$C^1$-structurally stable
diffeomorphisms} of compact manifolds ($K$ is then the whole
manifold). This description used the so-called concept of \emph{Axiom $A$ diffeomorphisms}: the diffeomorphisms for which the nonwandering set is hyperbolic and equal to the closure of the set of periodic points.  A diffeomorphism satisfies \emph{Axiom $A$ and the strong transversality condition} ($AS$) if moreover the stable and unstable manifolds of two nonwandering points intersect each other transversally. 

The works of Smale \cite{Sm}, Palis  \cite{PS}, de Melo  \cite{dM},
 Robbin  \cite{Ri}, Robinson  \cite{Rs} and Mañe \cite{Mane} 
achieved a satisfactory description of the $C^1$-structurally stable
diffeomorphisms stated in the following Theorem.

\begin{theo}[Structural Stability] The $C^1$-structurally stable diffeomorphisms of compact manifolds are exactly the $AS$ diffeomorphisms.\end{theo}

The concept of structural stability is fundamental since if one understands the global behavior of a structurally stable diffeomorphism, then he understands the topological behavior of every perturbation of it.
However such diffeomorphisms are not dense, and so this leads us to generalize this notion in order to include more diffeomorphisms.
A first generalization is the \emph{$\Omega$-stability}: every perturbation of the dynamics has a homeomorphic nonwandering set and, via this homeomorphism the dynamics on the nonwandering set is conjugated. However $\Omega$-stability does not imply that the interactions between the transitivity classes ({\it i.e.} basic set) are persistent and moreover is not generic in the $C^2$ topology (Newhouse phenomena). Also the $\Omega$-stable diffeomorphisms are Axiom $A$. Consequently every conservative $\Omega$-stable diffeomorphism is Anosov (the whole manifold is hyperbolic). This reduces a lot the variety of the examples.   

However Smale program was actually to understand most of the dynamics by splitting the phase space into several well understood invariant subsets. He simplified furthermore the study to surface diffeomorphisms and got the concept of Axiom A diffeomorphisms. I believe the following definition is a good candidate for the ``structurally stable subsets'' of the dynamics, since there are also non trivial examples in conservative dynamics.  

\begin{defi}
An invariant, compact subset $K$ of a $C^1$ diffeomorphism is $AS$ if:
 \begin{itemize}
\item[$(i)$]
The nonwandering set $\Lambda$ of the restriction $f|K$ is hyperbolic and equal to the closure of the periodic points in $K$:
\[\Lambda:= \Omega(f|K)= \overline{Per(f|K)},\]

\item[$(ii)$] There exists $\epsilon>0$ such that for all points $(x,y)\in K^2$, 
the $\epsilon$-local stable manifold $W^s_\epsilon(x)$ and the $\epsilon$-local unstable manifold $W^u_\epsilon(y)$ are transverse, and the (possibly empty) intersection  is included in $K$.
\end{itemize}
\end{defi}
We recall that the $\epsilon$-local stable manifold of $x$ is \[W^s_\epsilon(x):= \{y:\; d(f^n(x),f^n(y))<\epsilon, \forall n\ge 0\}.\]

\begin{exem}\label{premier expl}
Let $f$ be a diffeomorphism which leaves invariant two hyperbolic compact subsets $K_1$ and $K_2$. We suppose that $K_1$ and $K_2$ are the closure of periodic points subsets. Let $(W^s_{loc}(x))_{x\in K_1}$ be a continuous family of local stable manifolds of points of $K_1$; let $(W^u_{loc}(y))_{y\in K_2}$ be a continuous family of local unstable manifolds of points of $K_2$. We suppose that the intersection of   
$W^s_{loc}(x)$ with $W^u_{loc}(y)$ is transverse and compact for all $x\in K_1$ and $y\in K_2$. Then the union $K_{12}:=\cup_{(x,y)\in K_1\times K_2} W^s_{loc}(x)\cap W^u_{loc}(y)$  is a compact subset. Also $K:= K_1\cup K_2\cup \bigcup_{n\in \mathbb Z} f^n (K_{12})$ is an $AS$ compact subset. 

We notice that $K$ is not hyperbolic if the dimension of the stable directions of $K_1$ and $K_2$ are different. This is the case for perturbations of the conservative dynamics of the product of the Riemannian sphere with the real line  
												\[f:\; (z,t)\mapsto \Big(2 z, \frac{1+2|z|^2}{2+|z|^2} t\Big).\]
 For some perturbation $f$, the hyperbolic fixed point $K_1$ close to $(0,0)$ has a local unstable manifold that intersects transversally a local stable manifold of the hyperbolic fixed point $K_2$ close to $(\infty,0)$ at a circle $K_{12}$.          
\end{exem}

We will generalize in section \ref{1.1}   some of the dynamical properties of $AS$ diffeomorphisms to $AS$ compact subsets.

This is a first Theorem of this paper:
\begin{theo}\label{premier thm} If $K$ is an $AS$ compact subset for a $C^1$ diffeomorphism a manifold $M$ of dimension at most $2$, then $K$ is structurally stable.
\end{theo}

\begin{exem} Let $f$ be a Morse-Smale diffeomorphism of a surface $M$. This means that $f$ is an $AS$ diffeomorphism with a finite nonwandering set. 
Consequently there exists an attracting periodic orbit $(p_i)_{i=1}^n$. We suppose that the eigenvalues of the derivative of $f^n$ at this orbit are not real.

 As an algebraic geometer we blow up each point $p_i$ to a circle $\mathbb S_i$ and a neighborhood of it to a Möbius strip. This makes a new surface $\tilde M$. The dynamics $f$ lifts to a dynamics $\tilde f$ on $\tilde M$ which preserves the circle and acts on them as a rotation. 
 
 Moreover, the manifolds $\tilde M\setminus \cup_i \mathbb S_i$ and $M\setminus \cup_i \{p_i\}$ are equal, also   the respective restrictions of the dynamics $\tilde f$ and $f$ are equal. Therefore, the complement of the attraction basin of $\cup_i \mathbb S_i$ is an $AS$ compact subset of $\tilde M$.  
\end{exem}
\begin{rema} The above Theorem implies obviously the main Theorem of \cite{dM} on structural stability of $AS$ surface diffeomorphisms. Actually the latter article is partially wrong because the affirmation ``Clearly $h_g$ also preserves the foliation $\mathcal F_u$'' P 244 L-3 is wrong when the stable dimensions of the basic pieces $k$ and $j$ are not the same. On the other hand Theorem \ref{premier thm} does not generalize \cite{Rs}, but its proof is hopefully easier to understand...
\end{rema} 

Another generalization of the structural stability Theorem of a hyperbolic set $K$ is to blow up every point of $K$ to a manifold, in order to obtain a 
family of disjoint immersed manifolds which depend locally $C^1$ continuously. This reaches the concept of lamination. 

A \emph{lamination} is a secondly countable metric space $L$ locally modeled on the product of $\mathbb R^n$ with a locally compact metric space $T$ such that the changes of coordinates are of the form:
\[\phi_{ij}:\; U_i\subset \mathbb R^n\times T\rightarrow U_j\subset \mathbb R^n\times T'\]
\[(x,t)\mapsto (g(x,t),\psi(x,t)),\]
where the partial derivative w.r.t. $x$ of $g$ exists and is a continuous function of both $x$ and $t$, and where $\psi$ is locally constant w.r.t. $x$. 

A \emph{plaque} is a set of the form $\phi^{-1}(\mathbb R^n\times \{t\})$ where $\phi$ is a chart. 
The \emph{leaf} of $x\in L$ is the union of all the plaques which contain $x$.
A maximal atlas  $\mathcal L$ of compatible charts is a \emph{lamination structure} on $L$. Given an open subset $L'$ of $L$, we denote by $\mathcal L{|L'}$ the structure of lamination on $L'$ formed by the charts of $\mathcal L$ whose domain is included in $L'$.

 The reader not familiar with the laminations should look at \cite{ghys}, \cite{berlam}.

Laminations are specially interesting when they are embedded. An \emph{embedding} of a lamination $(L,\mathcal L)$ into a manifold $M$ is a homeomorphism $i$ onto its image which is an \emph{immersion}. A continuous map $i:\; L\rightarrow M$ is an \emph{immersion} if its differential along the plaques of $\mathcal L$ exists, is injective and depends continuously on $x\in L$. Two embeddings are \emph{$C^1$ close} if they are close in the $C^0$-compact-open topology  and their differential along the plaques are close for the compact-open topology. The reader might see \cite{berlam} for more details about this topology.    

Usually we identify an embedded lamination with its image. We note that its plaques are submanifolds of $M$ and its leaves form a family of disjoint injectively immersed  submanifolds. Thus the tangent space $T_x\mathcal L$ at $x\in L$ of its leaf is identified to a subspace of the tangent space $T_xM$ of $M$ at $x$.

 A diffeomorphism $f$ of $M$ \emph{preserves} this lamination if it sends each leaf of $\mathcal L$ into a leaf of $\mathcal L$, or equivalently each plaque of $\mathcal L$ into a plaque of $\mathcal L$. 

Such a lamination is $C^1$-persistent if for $f'$ $C^1$ close to $f$ there exists an embedding $i'$ $C^1$ close to $i$ such that $f'$ preserves $\mathcal L$ embedded by $i'$ and induces the same dynamics on the leaves of $\mathcal L$ as $f$. We notice that when the \emph{dimension} of the lamination (that is of its leaves) is zero, then the lamination $\mathcal L$ is persistent {\it iff} the subset $L$ is structurally stable.

Let us recall that a diffeomorphism $f$ is \emph{normally hyperbolic} to an embedded lamination $(L,\mathcal L)$ if 
the tangent
bundle of $M$ restricted to $L$ is the direct sum of $T\mathcal L$ with two $Tf$ invariant subbundles $E^s\rightarrow L$ and  $E^u\rightarrow L$ such that the following property holds. There exists $\lambda<1$ such that for every $x\in L$, for every $n\ge 0$ sufficiently large, for  all unit vectors $v_c\in T_x\mathcal L$,
$v_s\in E^s_x$ and $v_u\in E^u_x$:
\[  \|T_xf^n (v_s)\|<\lambda^n \min(1,\|T_xf^n(v_c)\|)\quad \mathrm{and}\quad  \lambda^n\|T_xf^n (v_u)\|> \max(1,\|T_xf^n(v_c)\|).
  \]
When $L$ is compact, $n$ can be chosen independently of $x$. Otherwise this is in general not the case.
If the strong stable direction $E^s$ is $0$, then the lamination $(L,\mathcal L)$ is \emph{normally expanded} by $f$.

The founders of this way of generalization were Hirsch, Pugh and Shub (HPS). This is their main Theorem:
\begin{theo}[Hirsch-Pugh-Shub \cite{HPS}]\label{HPS}
Compact, normally hyperbolic and plaque-expansive laminations are persistent.
\end{theo}
This Theorem is the fundamental one of the partially hyperbolic dynamics field.

The plaque-expansiveness is a generalization of the expansiveness to the space of leaves. The necessity and automaticity of the plaque-expansiveness in the above Theorem is still an open problem.
We will give its  definition in Section \ref{plaque-expansive}.

 A paradigmatic application of the HPS's Theorem is when $f$ is the product dynamics of an Anosov diffeomorphism of a compact manifold $M$ with the identity of a compact manifold $N$. Then \emph{the bundle $M\times N\rightarrow M$ is $C^1$-persistent as a lamination}. This means that for every $f'$ $C^1$ close to $f$, there exists a continuous family of disjoint $C^1$ submanifolds $(N'_x)_{x\in M}$ s.t. $N_x'$ is $C^1$ close to $\{x\}\times N$ and $f'(N_x')=N_{f'(x)}$.

Let us generalize the above example by considering that  $f$ is the product dynamics of an $AS$ diffeomorphism $g$ of $M$ with the identity of a manifold $N$. Then the persistence of $\mathcal L$ is not a consequence of the HPS' Theorem if $g$ is not Anosov nor a consequence of the structural stability Theorem if $N$ has non-zero dimension. However it is the consequence of the main Theorem of this paper if the dimension of $M$ is $2$ (the case of dimension $1$ is easy).

\begin{theo}[Main result]\label{Main result} Let $\pi:\; \hat M\rightarrow M$ be a $C^1$ bundle over a surface with compact fibers. 
Let $K$ be an $AS$ compact subset for a diffeomorphism $f$ of $M$. Let $\hat f$ be a diffeomorphism of $\hat M$ which preserves the bundle $\pi$ and lift $f$. We suppose that the bundle is normally hyperbolic over
the nonwandering set of $f|K$. Then the bundle over $K$ is $C^1$-persistent as a lamination. In other words the lamination $(L,\mathcal L)$ supported by $\pi^{-1}(K)$ and whose leaves are the connected component of fibers of $K$'s points is $C^1$-persistent.\end{theo}

A simple application of this Theorem is when $K$ is equal to $M$ and so $f$ is an $AS$ diffeomorphism of a surface; $\hat f$ is the product dynamics on $\hat M=M\times N$ of $f$ with the identity on a compact manifold $N$.  

\begin{exem} Let us come back to the dynamics $f:\;(z,t)\mapsto \big(4z, \frac{1+2|z|^2}{2+|z|^2} t\big)$ of the product of the sphere with the real line in Example \ref{premier expl}. This is an $AS$ bundle over the sphere. In order to use
 Theorem \ref{Main result}, we compactify canonically the fibers to circles. Such a compact bundle is persistent. This dynamics appears in some hetero-dimensional cycles of dynamics of $\mathbb R^3$.
For instance when the cycle is supported by the union of the unit sphere with the vertical line passing through the poles $\{0,\infty\}$ of the sphere, and whose dynamics at the neighborhood of unit sphere is $f$. The persistence of this bundle might be useful for showing some non-uniform hyperbolic properties of perturbations of this hetero-dimensional cycle.\end{exem}    

We notice that contrary to the previous Theorems, the hypothesis of this one is not open. Actually the proof will exhibit a class of laminations which is open, but too complicated to be described here. Though the above Theorem generalizes the structural stability Theorem in dimension two, it does not generalize the HPS Theorem. In order to state a conjecture generalizing both results, let us recall the definition of a saturated set.

A \emph{saturated set} $\Lambda$ of a lamination $(L,\mathcal L)$ embedded into a manifold $M$ is a union of leaves of $\mathcal L$. If $K$ is a subset of $L$, the $\mathcal L$-saturated set of $K$ is the union of the leaves of $K$'s points. We note that if $K$ is an invariant compact subset and if $L$ is compact, then its $\mathcal L$-saturated set is an invariant compact subset.

\begin{defi}A compact lamination $(L,\mathcal L)$ embedded into a manifold $M$ and preserved by a diffeomorphism $f$ is \emph{normally $AS$} if there exists $\epsilon>0$ such that:
\begin{itemize}
 \item The saturated $\Omega(\mathcal L)$ of the nonwandering set of $f$ restricted to $L$ is  normally hyperbolic (and plaque-expansive),
 \item For any pair of $\mathcal L$ leaves, the $\epsilon$-local stable set of the one (which is an immersed submanifold) is transverse to the $\epsilon$-local unstable set 
of the other, and their intersection is included in $L$,
\item $\Omega(\mathcal L)$ is locally maximal in $M$.
\end{itemize}
\end{defi}

This is our conjecture:
\begin{conj}\label{3}
Compact normally $AS$ laminations are $C^1$-persistent.
\end{conj}

We notice that the main Theorem solves this conjecture when the leaves of the foliation are the connected components of a preserved $C^1$ bundle over a surface.

Physically we can interpret this conjecture as follows. The modeling  of a phenomenon neglects usually many variables. These variables are geometrically the fibers of a bundle, or even the leaves of a lamination. The dynamics along these fibers is weaker than the one on the basis, thus an $AS$  compact subset is blown up to an $AS$ lamination.
Thus we propose here to show that if we do not neglect these variables and perturb the full system, then the new system is conjugated to first modulo new negligible variables.
\newline

Half of this work has been established during my PhD thesis under the direction of J.-C. Yoccoz at Université Paris Sud. I am very grateful for his guidance. The rest of this work was done at the CRM (Bellatera, Spain), the IHÉS (Bures-sur-Yvette, France) and the LAGA (Paris, France).

\section{Geometry and dynamics of $AS$ compact subsets}

Axiom $A$ and $AS$ diffeomorphisms were deeply studied in the 60-70's. We are going to generalize some of these results to  $AS$ compact subsets. 

Let $f$ be a diffeomorphism of a compact manifold $M$ and let $K$ be an $AS$ compact subset $f$-invariant. Let $\Lambda$ be the nonwandering set of the restriction of $f$ to $K$.

For $x\in\Lambda$, we denote by $W^s_K(x)$ the intersection of the stable manifold of $x$ with $K$; we denote by $ W^s_{K\epsilon}(x)$ the union of the $\epsilon$-local stable  manifolds of points of $W^s_K(x)$. We define similarly $W^u_K(x)$ and $W^u_{K\epsilon}(x)$.

Let $\lambda<1$ be greater than the contraction of the stable direction of $\Lambda$.  Let $\tilde d:\; (x,y)\mapsto \sup_{n\ge 0} \min\big(\frac{d(f^n(x),f^n(y))}{\lambda^n}, 1)$. 
For $x\in K$ and $\epsilon\in (0,1)$, let  $\tilde W^s_\epsilon(x)$ be the ball centered at $x$ and with radius $\epsilon$ for this metric $\tilde d$. 

We notice the following property:
\begin{fact}\label{factpedago} For every $x\in K$, the ball $\tilde W^s_\epsilon$ is a local stable manifold of $x$. Moreover:
\begin{itemize} 
\item $f\big(\tilde W^s_\epsilon(x)\big)\subset  \tilde W^s_{\lambda\epsilon}\big(f(x)\big)$,
\item there exists $n>0$ such that $\tilde W_\epsilon^s(x)=\cap_{m\le n} f^{-m}\Big(W^s_{\lambda^m\epsilon } \big(f^m(x)\big)\Big)$,
\item $cl(\tilde W^s_{\epsilon}(x))=\cap_{\epsilon'>\epsilon} \tilde W^s_{\epsilon'}(x)$.\end{itemize}
\end{fact}
\begin{proof} The first item is clear and the last is an immediate consequence of the second. To show the second we notice that the point $x$ has its iterate $f^n(x)$ which belongs to $\cup_{y \in \Lambda} W^s_\epsilon(y)$ for some large $n$, and then $W^s_\epsilon(f^n (x))=\tilde W^s_\epsilon(f^n (x))$.\end{proof}

We call $\tilde W^s_\epsilon(x)$ \emph{the adapted $\epsilon$-local stable manifold of $x$}. 

The adapted $\epsilon$-local stable manifolds are inspired from the work of Robinson: their geometry is similar to the one of ``unstable disk families'' well described in \cite{Robgeo}.

\subsection{Dynamics on K}\label{1.1}
Let us describe the dynamics on $K$ by adapting classical arguments (see \cite{Sm}, \cite{Sh}).

If for two periodic points $x,y\in\Lambda$, the set  $W^s_{K}(x)$ intersects $W^u_{K}(y)$ and $ W^s_{K}(y)$ intersects $W^u_{K}(x)$, we write $x\sim y$. This relation is obviously reflexive and symmetric. Let us show its transitivity.
 
First we note that two points in relation  $\sim$ have their stable manifold (resp. unstable manifold) of the same dimension. Moreover, the relation is $f$-invariant:  $x\sim y$ iff $f(x)\sim f(y)$.

Let $x$, $y$ and $z$ be three periodic points of $\Lambda$ such that $x\sim y$ and $y\sim z$. 
By the $f$-invariance of $\sim$, the points $x$ and $z$ are in relation for $f$ if their images by any iterates of $f$ are in relation. So  we can suppose that $x$, $y$ and $z$ are fixed points of $f$.
Let $y_x$ and $y_z$ be two points of the intersections $W^s_K(x)\cap W^u_K(y)$ and $W^s_K(y)\cap W^u_K(z)$ respectively. 
By iterating the dynamics, we may assume that $y_x$ and $y_z$ are close to $y$. By the lambda lemma, $W^s_\epsilon(y_x)$ and  $W^u_\epsilon(y_z)$ are close to $W^s_\epsilon(y)$ and  $W^u_\epsilon(y)$ respectively. Thus the intersection $W^s_\epsilon(y_x)\cap W^u_\epsilon(y_z)\subset W^s_{K\epsilon}(x)\cap W^u_{K\epsilon}(z)$  is not empty.
And so by Property $ii$ of the $AS$ definition, the intersection
 $W^u_K(x)\cap W^s_K(z)$ is non empty. 
Such an argument proves that $x\sim z$.  Thus the relation $\sim$ is an equivalence relation.

As two close  points of $ \Lambda\cap Per(f)$ are equivalent; the sets supporting two different equivalence classes of $\sim$ are $\delta$-distant for some $\delta>0$. 

We denote by $(\Lambda_{i,j})_{i,j}$ the closure of the equivalence classes in $\Lambda$.  
They are $\delta$-distant and their union is equal to the compact set $\Lambda$, by density of the periodic points. Thus there are finitely many equivalence classes. By $f$ invariance of $\sim$, these equivalence classes are all periodic. We put $f(\Lambda_{i,j})=\Lambda_{i,j+1}$ and $\Lambda_i:= \cup_j \Lambda_{i,j}$.  

The family $(\Lambda_i)_i$ is the \emph{spectral decomposition} of $\Lambda$ and each of its elements is a \emph{basic set}.

Let us show that each basic set $\Lambda_i$ is transitive: for all $x,y\in \Lambda_i$, for all neighborhoods $U$ of $x$ and $V$ of $y$ for the topology of $K$, there exists $n>0$ such that $f^n(U)\cap V$ is nonempty. By density of the periodic points, we only need to show this when $x$ and $y$ are periodic. By replacing $y$ by some of its iterates, we may suppose that $x$ and $y$ belong to a same $\Lambda_{i,j}$. Let $q\in W_K^u(x)\cap W_K^s(y)$. For every sufficiently large multiple $n $ of the period of $x$ and $y$, $f^{-n}(q)$ belongs to $U$ and $f^{n}(q)$ belongs to $V$. Thus $f^{2n}(U)$ intersects $V$.

As each basic set $\Lambda_i$ is hyperbolic and in the closure of the periodic orbits, the stable set of $\Lambda_i$:
\[W^s(\Lambda_i):=\{x\in M:\; d(f^n(x),\Lambda_i)\rightarrow 0,\quad \mathrm{when}\; x\rightarrow +\infty\}\]
is equal to the union of manifolds $\cup_{x\in \Lambda_i} W^s(x)$.
The reader should look at \cite{Sh}, Prop 9.1 for a proof. 

Intersecting with $K$, we get: 
\[W^s_i:=\{x\in K:\; d(f^n(x),\Lambda_i)\rightarrow 0,\quad \mathrm{when}\; x\rightarrow +\infty\}=\bigcup_{x\in \Lambda_i} W^s_K(x).\]
Similarly:
\[W^u_i:=\{x\in K:\; d(f^n(x),\Lambda_i)\rightarrow 0,\quad \mathrm{when}\; x\rightarrow -\infty\}=\bigcup_{x\in \Lambda_i} W^u_K(x).\]

Let us write $\Lambda_i\succ \Lambda_j$ if $W^u_i\setminus \Lambda_i$ intersects $W^s_j\setminus \Lambda_j$. 


\begin{prop} If $\Lambda_i\succ \Lambda_j$ then there exist two periodic points $p_i\in\Lambda_i$ and $p_j\in\Lambda_j$ such that $W^u(p_i)$ intersects $ W^s(p_j)$.\end{prop}
\begin{proof}
Let $q_i\in\Lambda_i$ and $q_j\in\Lambda_j$ be two points such that 
$W^u(q_i)$ intersects $ W^s(q_j)$ at some point $q\in K$. 
Thus for $n>0$ large, the iterate $f^n(q)$ is close to $f^n(q_j)$. Consequently the $\epsilon$-local unstable manifold of $f^n(q)$ intersects the $\epsilon$-local stable manifold of a periodic point $f^n(p_j)\in \Lambda_j$  close to $f^n(q_j)$. Let $f^n(p')$ be this intersection point. It must belong to $K$. Pulling back this construction by $f^n$, we get the existence of an intersection point $p' \in K$ between $W^s(q_i)$ and $W^u(p_j)$. We construct similarly $p_i$.
\end{proof}
\begin{prop}
The relation $\succ$ can be completed to a total order.
\end{prop}

\begin{proof}
This Proposition is equivalent to the no-cycle condition: if $\Lambda_{j+1}\succeq \Lambda_j$ for $j\in \{1,\dots , n\}$ with $\Lambda_{n+1}=\Lambda_1$, then 
$\Lambda_j=\Lambda_{j+1}$ for every $j\in \{1,\dots, n\}$.
Let us construct by induction on $i$, a periodic point $p_i\in \Lambda_i$ such that $W^u_K(p_i)$ intersects $W^s_K(p_{i+1})$ at a point $q_i\in K$ and for $i=n$,  $f^k(p_{n+1})\sim p_1$ for some $k$.

The existence of $p_1$ and $p_2$ follows from the last Proposition. Let us suppose $(p_l)_{l\le i}$ constructed. There exist periodic points 
$p_i'\in \Lambda_i$ and 
$p_{i+1}\in \Lambda_{i+1}$ such that $W^u_K(p_i')$ intersects $W^s_K(p_{i+1})$. By replacing the points $p_i'$ and $p_{i+1}$ by their image by some iterate $f^k$, we may assume that $p_i\sim p_i'$. Thus, by proceeding as for the proof of the transitivity of $\sim$, we get that $W^u_K(p_i)$ intersects $W^s_K(p_{i+1})$. The last statement for $i=n$ is clear.

One easily shows that each point $q_i\in W^u_K(p_i)\cap W^s_K(p_{i+1})$ is nonwandering and so belongs to $\Lambda$. 
Looking at the backward and forward orbit of $q_i$, this point belongs to a basic set which is arbitrarily close to $\Lambda_i$ and $\Lambda_{i+1}$. So indeed the basic pieces $(\Lambda_i)_{i=1}^n$ coincide.  
\end{proof}

 We renumber the sets $(\Lambda_i)_i$ according to this total order, so $\Lambda_i\prec \Lambda_j$ only if $i< j$. 
By the above Proposition, we note that if $i\le j$ then $W^u _i$ does not intersect $W^s_j$. However $W^s_i$ does not necessarily intersect $W^u_j$.   
 
\begin{theo} Let $f$ be a homeomorphism of a metric space which leaves invariant a compact subset $K$. Let $\Lambda= \Lambda_1\sqcup \cdots \sqcup \Lambda_N$ be a disjoint union of invariant closed subsets which contains the ($\omega$ and $\alpha$)-limit set of $f|K$. If for every $i< j$, $W^u_i\cap W^s_j=\emptyset$ then 
 there exists of a sequence of closed subsets of $K$:
 \[\emptyset=K_{0}\subset K_1\subset \cdots \subset K_{N}=K\]
  such that for all $i\ge 1$:

\begin{enumerate}
\item  $f(K_i)\subset \ring K_i$, with $\ring K_i$ the interior of $K_i$ for the topology of $K$ induced by $M$,  
\item $\Lambda_i=\bigcap_{n\in \mathbb Z}f^n(K_i\setminus K_{i-1})$.
\end{enumerate}
\end{theo}
Such a sequence $(K_i)_i$ is called a \emph{filtration adapted to} $(\Lambda_i)_i$. 

For a proof the reader should read Theorem 2.3 of \cite{Sh}.  Actually the cited Theorem asks $K$ to be a manifold, but this fact is useful uniquely to choose $(K_i)_i$ among the submanifolds with boundary. This is not asked here.

\begin{Property}\label{propfiltration}
For every $i\in \{0,\dots , N\}$, the following equality holds : $\cap_{n\ge0} f^n (K_i)= \cup_{j\le i} W^u_j$.
\end{Property}
\begin{proof}
The compact set $\cap_{n\ge 0} f^n (K_i)$ contains its $f$-preimage and so its backward limit set is included in $\sqcup_{j\le i} \Lambda_j$. Thus  $\cap_{n\ge0} f^n (K_i)$ is included in $\cup_{j\le i} W^u_j$.

 On the other hand, $\cap_{n \ge 0} f^n(K_i)$ contains $\sqcup_{j\le i} W^u_\epsilon(\Lambda_j)\cap K$ and by $f^{-1}$-stability it contains $\cup_{j\le i} W^u_j$. \end{proof} 
 
Let us recall a last classical definition: 

A \emph{fundamental domain} for the unstable manifolds of a hyperbolic set $\Lambda$, is a closed subset $\Delta\subset W^u(\Lambda)$ such that
\[\bigcup_{n\in \mathbb Z} f^n(D)\supset W^u(\Lambda)\setminus \Lambda.\]
The fundamental domain $D$ is \emph{proper} if moreover $\Delta\cap \Lambda$ is empty.

Let us recall the

\begin{lemm}[\cite{HPSP} Lemma 3.2]\label{HPSP}
Let $\Lambda$ be a hyperbolic set for $f$ such that $cl\big(Per(f)\cap \Lambda\big)=\Lambda$. For every $\delta$ small enough:
\[\Delta:= cl\Big(W_\delta^u(\Lambda)\setminus  f^{-1}\big(W_\delta^u(\Lambda)\big)\Big)\]
is a proper fundamental domain for $W^u(\Lambda)$.

\end{lemm}

\subsection{Stratified structures on a local stable set of $K$}

We are going to endow a local stable set of $K$ with a structure of stratification of laminations, and then with a structure of trellis of laminations. The main steps of this construction are illustrated in figure \ref{constructionfinal}.
\begin{figure}
        \includegraphics{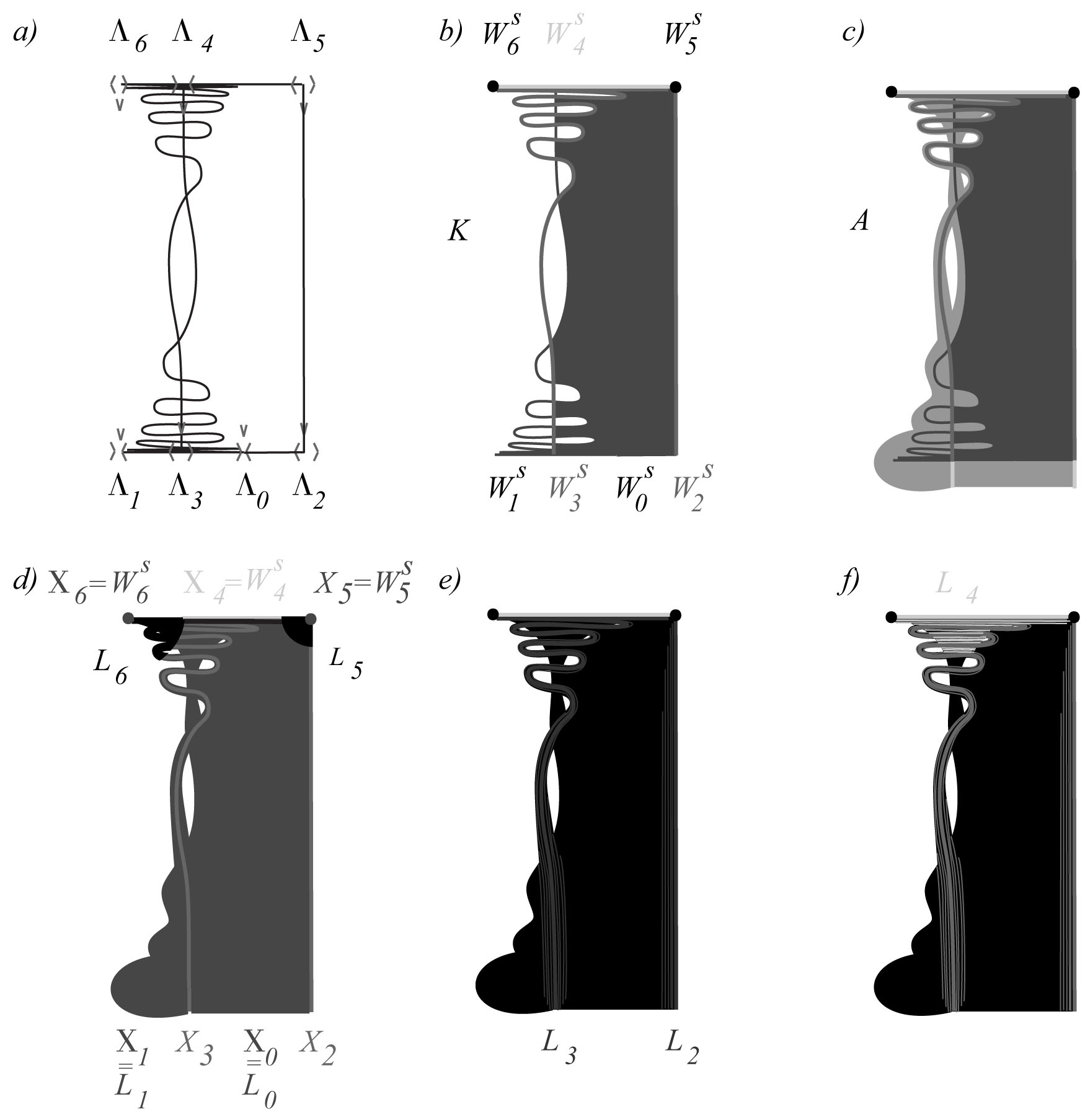}
        
    \caption{{\small $a)$ Morse-Smale Dynamics of the plan whose nonwandering set is equal to seven fixed points (2 sinks $\Lambda_0$ and $\Lambda_1$, 3 saddles $\Lambda_2$, $\Lambda_3$ and $\Lambda_4$, and also 2 sources $\Lambda_5$ and $\Lambda_6$). 
    $b)$  The compact set $K$ is equal to the union of $\Lambda_0$, $\Lambda_1$, $\Lambda_2$, $\Lambda_5$, $\Lambda_6$, $W^s(\Lambda_4)$, $W^u(\Lambda_3)$; a half stable manifold of $\Lambda_3$ and half unstable manifolds of $\Lambda_2$ and $\Lambda_4$, and also the domain bounded by these five points and curves. $c)$ Local stable set $A$ of $K$. $d)$ Stratification of laminations  $\Sigma$ on $A$ and the tubular neighborhoods of the strata of dimension different to one. $e)$ Tubular neighborhoods of the highest one dimensional strata. $f)$ Tubular neighborhood of a lower one dimensional stratum.}}
    \label{constructionfinal}
\end{figure}

Following J. Mather \cite{Ma}, a \emph{stratified space} is a metric space $A$ equipped with a finite partition $\Sigma$ of $A$ into locally closed subsets, satisfying
the \emph{axiom of the frontier}:
  \[\forall (X,Y)\in \Sigma^2,\; cl(X)\cap Y\not=\emptyset\Rightarrow cl(X)\supset Y.\qquad \mathrm{We\; write\; then}\; X\ge Y.\]
The pair $(A,\Sigma)$ is called \emph{stratified space} with \emph{support} $A$ and \emph{stratification} $\Sigma$.\\

Let us illustrate this definition by a useful Proposition illustrated by figure \ref{constructionfinal}. b), 
\begin{prop}\label{prop1} Let $f$ be a diffeomorphism of a manifold (of any dimension). 
Let $K$ be an $AS$ compact subset of $M$. Let $(\Lambda_i)_i$ be the basic set of $f|K$ and $W^s_i:=W^s(\Lambda_i)\cap K$. Then the family of sets $(W^s_i)_i$ forms a stratification on $K$.

Moreover, $W^s_i\ge W_j^s$ implies $\Lambda_i\preceq \Lambda_j$ and so $i\le j$.   
  
\end{prop}
\begin{proof} The family $(W^s_i)_i$ is obviously a finite partition. Let us show the frontier condition and the relation between $\ge $ and $\preceq$. Let $i,j$ be such that $ cl(W^s_i)$ intersects $W_j^s$.

 If $i$ is not equal to $j$, it follows from Lemmas 1 and 2 P.10 of \cite{Sh} that the closure of $W^s_i$ intersects $ W^u_j\setminus \Lambda_j$. This gives the last statement of the Proposition. Let $x$ be a point which belongs to this intersection. As $(W^s_k)_k$ covers $K$, there exists $j_1$ such that $x$ belongs to $W^s_{j_1}$. By Property $(ii)$ of $AS$ compact set and the lambda lemma, the closure of $W^s_{j_1}$ contains $W^s_{j}$. Moreover, the closure of $W^s_{i}$ intersects $W^s_{j_1}$. And so on, we can continue to construct $(\Lambda_{j_k})_k$. As the family $(\Lambda_i)_i$ is finite and there is no-cycle, the family $(\Lambda_{j_k})_k$ is finite. Thus, we obtain:
\[cl(W^s_i)=cl(W^s_{j_n})\supset \cdots \supset cl(W^s_{j_1})\supset W^s_j.\]

\end{proof}

For persistence purpose, it is useful to endow the strata with a geometric structure. When a hyperbolic set $\Lambda$ is equal to the closure of a periodic points subset, then its has a \emph{local product structure} (see prop. 8.11 of \cite{Sh}) and so $W^s_\epsilon (\Lambda):=\cup_{x\in \Lambda} W^s_\epsilon (x)$ has a structure of lamination whose plaques are local stable manifolds.

A \emph{stratification of laminations} is the data of a laminar structure on each stratum of a stratification, such that for any strata $X\ge Y$, the dimension of $X$ is at least equal to the dimension of $Y$.


Under the hypotheses of Propositon \ref{prop1} with furthermore $K$ equal to the whole manifold $M$, 
(that is $K=M$ and $f$ is $AS$), then 
we have proved in Proposition 1.2.7  of \cite{Bermem} that each stable set $W^s_i$ has a lamination structure whose leaves are the local stable manifolds of $\Lambda_i$´s points. Moreover the laminations $(W^s_i)_i$ form the strata of a stratification of laminations $\Sigma$ on $K=M$.

When $K$ is not $M$, unfortunately $(W^s_i)_i$ is in general not a stratification of laminations. As we see on figure \ref{constructionfinal}.b) the strata $W^s_0$ and $W^s_1$ correspond to the sinks basins, but are not open in $M$ and so cannot have a two dimensional laminar structure. This is why we extend $K$ to the union of the adapted stable manifolds of its points, as illustrated by figure \ref{constructionfinal}.c).
 
\begin{prop} \label{prop2}
Let $f$ be a diffeomorphism of a manifold (of any dimension). 
Let $K$ be an $AS$ compact subset of $M$. 

For every small $\epsilon>0$, there exists a stratification of laminations $\Sigma$ on 
$A:= \cup_{x\in K} \tilde W^s_\epsilon(x)$.  Each stratum $X_k$ of $\Sigma$ is associated to a basic set $\Lambda_k$. The support of $X_k$ is 
\[X_k:=\bigcup_{x\in W^s_k} \tilde W^s_\epsilon(x)=W^s(\Lambda_k)\cap A.\]
 The plaques of $X_k$ are local stable manifolds. 
\end{prop}
\begin{proof}
There are two properties to be proved. The first is the frontier condition of the stratification, the second is the laminar structure on each of these strata. 

\noindent\textbf{Frontier condition of $\Sigma$.}\\
Let $cl(X_i)\cap X_j\not=\emptyset$.  In other words, the closure of $X_i=W^s(\Lambda_i)\cap A$ intersects $X_j=W^s(\Lambda_j)\cap A$.
Let $x$ be a point of this intersection and let $x_n:= f^n(x)$. Then any forward limit $z$ of $(x_n)_n$ belongs to $\Lambda_j\subset W^s_j$. 
Also there exists $(y^k)_k\in W_i^{s\mathbb N}$ such that $x$ is close to $\tilde W^s_\epsilon(y^k)$ when $k$ is large. Thus $x_n$ is close to $\tilde W^s_{\lambda^n\epsilon}(y^k_n)$, with $y^k_n=f^n(y^k)$ by fact \ref{factpedago}.  Taking first $n$ large and then $k$ large w.r.t. $n$, we get that $x_n$ is close to $W^s_i$. This implies that $z$ belongs to $ cl(W^s_i)$.   
We proved that the closure of $W^s_i$  intersects $\Lambda_j\subset W^s_j$. By Proposition \ref{prop1}, the closure of $W^s_i$ contains $W^s_j$ and $W^s_i$ intersects $ W^u_j$. 

We want to prove $X_i\ge  X_j$. This is now a mere consequence of the following Lemma:

\begin{lemm} The map $x\in A\mapsto cl\big(\tilde W_\epsilon^s(x)\big)$ is continuous for the Hausdorff topology.

Moreover for every $i$, the map  $x\in X_i\mapsto \tilde W_\epsilon^s(x)$ is continuous for the $C^1$-topology.
\label{texnic}\end{lemm}
\begin{proof}

We start by proving the upper semi-continuity of the map: $\tilde W^s_\epsilon(y)$ is included in a small neighborhood of $\tilde W^s_\epsilon(x)$, when $y\in A$ is close to $x\in A$.

Let $(y_n)_n$ be a sequence which converges to $x$. For each $n$, chose $y'_n\in \tilde W^s_\epsilon(y_n)$ and take  a limit point $y'$ of $(y'_n)_n$ when $n$ approaches infinity. 
For every $\epsilon'>\epsilon$, we note that $y'$ belongs to the $\epsilon'$-ball centered at $\lim y_n=x$. Also $f^k(y')$ belongs to the $\lambda^k\epsilon'$-ball centered at $\lim f^k (y_n)=f^k(x)$. Thus $y'$ belongs to  $\cap_{\epsilon'>\epsilon}\tilde W^s_{\epsilon'}(x)=cl\big(\tilde W^s_\epsilon(x)\big)$.

The lower semi-continuity is more tricky. We want to show that for every $x\in A$, every $\eta>0$, every $y\in A$ close to $x$, the  $\eta$-neighborhood of $\tilde W_\epsilon^s(y)$ contains $\tilde W_\epsilon^s(x)$.

We observe that the adapted local stable manifolds depend continuously on each set $X_i$. Indeed,  
for $x$ and $y$ in $X_i$, for $n$ large, the points $\{f^k(x)\}_{k\ge n}$ and $\{f^k(y)\}_{k\ge n}$ are nearby $\Lambda_i$ and since $\lambda$ is weaker than the contraction of the basic sets, we have: 
\[\tilde W_\epsilon^s(x)=\bigcap_{m\le n} f^{-m}\Big(W^s_{\lambda^m\epsilon } \big(f^m(x)\big)\Big)\quad \mathrm{and}\quad \tilde W_\epsilon^s(y)=\bigcap_{m\le n} f^{-m}\Big(W^s_{\lambda^m\epsilon } \big(f^m(y)\big)\Big).\]
This proves the second statement of the Lemma.

 Thus the following Lemma accomplishes the proof of the continuity:
 \begin{lemm}\label{champ de plaque elementaire} For every   $y\in X_i$ close to $x\in X_j$ with $i< j$,  the manifold $\tilde W^s_\epsilon(x)$ contains a submanifold $C^1$ close to $\tilde W^s_\epsilon(y)$. \end{lemm}

We first prove this fact for $x,y\in K$ because it sufficient to prove that $(X_k)_k$ is a stratification. 

Thus, we begin with $x\in W^s_j$ and $y\in W^s_i$,  where $W^s_i>W^s_j$.
 
By transitivity of the frontier condition of the stratification $(W^s_k)_k$, it is sufficient to prove the case for which there is no $k$ such that $W^s_i> W^s_k> W^s_j$. 


By replacing $x$ and $y$ by their image by an iterate  $f^n$ of  the dynamics and $\epsilon$ by $\lambda^n\epsilon$, we can suppose these points close to $\Lambda_j$ and $\epsilon$ small, since $\tilde W^s_\epsilon(z)= \cap_{k=0}^n f^{-k}\big(B(f^k(z),\epsilon)\big)\cap f^{-n} \big(\tilde W^s_{\lambda^n\epsilon}(f^n(z))\big)$, for all $z\in A$ and $n\ge 0$.
%
 
Let $C$ be a cone field on a neighborhood $U$ of $W^s_\epsilon (\Lambda_j)$, which satisfies the following properties:
\begin{itemize}
\item for $z\in f(U)\cap U$, $Tf^{-1}$ sends $C(z)$ into $C(f^{-1}(z))$ and $\lambda$-expands its vectors,
\item for $z\in W^s_\epsilon(\Lambda_j)\cap U$, the space $T_z W^s_\epsilon(z)$ is a maximal plane included in $C(z)$, 
\item $z\mapsto C(z)$ is continuous.
\end{itemize}

\begin{sublemma} 
The intersection of  $D:= cl\Big(W^u_\eta(\Lambda_j)\setminus f^{-1}\big(W^u_\eta(\Lambda_j)\big)\Big)$ with $W^s_i$ is compact.\end{sublemma}
\begin{proof} We notice that $D$ is disjoint from $W^s_j$
by Property \ref{propfiltration}, and so disjoint from 
$\cup_{l\le j} W^s_l$, whereas $W^s_i$ has its closure included in $W^s_i\cup \bigcup_{l\le j} W^s_l$.\end{proof}
    
\begin{sublemma} 
for $\eta>0$ small enough and $z\in D\cap W^s_i$, the tangent space $T_zW^s_\epsilon(z)$ intersects $C(z)$ at a subspace of the same dimension as the stable one of $\Lambda_j$. 
\end{sublemma} 
\begin{proof}
Otherwise there is $z\in W^s_i\cap W^u_\epsilon(\Lambda_j)$ arbitrarily close to $\Lambda_j$ such that the vectors of $T_z W^s(z)$ and $C(z)$ does not span the $T_z M$:
$T_z W^u(z)+ C(z)\not = T_z M.$
As $z$ is close to $\Lambda_j$ we can push forward this sum, and the limit exhibit a tangency between $W^u_\epsilon(\Lambda_j)$ and $X_i$. This is a contradiction with the definition of $AS$ compact set.    \end{proof}

By compactness, there exists $\delta>0$ such that for every $z$ in a compact neighborhood $N$ of $D\cap W^s_j$ in $X_i$, there exists a submanifold $P_0(z)$ of $\tilde W^s_\epsilon (z)$ such that:
\begin{itemize}
\item The tangent space of $P_0(z)$ is contained in $C(z)$ and its dimension is the same as the stable one of $\Lambda_j$.
\item $P_0(z)$ contains $z$ and its boundary is $\delta$-distant from $z$.
\end{itemize}
    
We suppose $\delta<\epsilon$. We define by induction $P_{n+1}(z):=U\cap  B(f^{-n-1}(z),\epsilon)\cap f^{-1}(P_n(z))$.
  
By the lambda lemma, $P_n(z)$ is $C^1$-close to a some $\epsilon$-local stable manifolds of $X_j$, if $f^{-n}(z)$ is close to $\Lambda_j$. 

By compactness, this holds uniformly for $z\in N$. such that  $f^{-n}(z)$ is close to $\Lambda_j$.
As $\lambda$ is greater than the contraction of $C$, we have $P_{n}(z)\subset \tilde W^s_\epsilon(f^{-n}(z))$.

As a point $y\in X_j$ nearby $x\in W^s_i$ belongs to $f^{-n}(N)$ with $n$ large, the adapted stable manifold $\tilde W^s_\epsilon(y)$ contains $P_n(f^n(y))$ which is close to $W^s_\epsilon(x)=\tilde W^s_\epsilon(x)$.

We proved that $x\in K\mapsto \tilde W_\epsilon^s(x)$ is continuous. This is actually sufficient to imply the frontier condition of $(X_k)_k$. 

Let us now reapply the argument to finish the lower-semi-continuity proof of $x\in A\mapsto \tilde W_\epsilon^s(x)$. It remains the case $x\in X_j$ and $y\in X_i$ with $i< j$. 
By transitivity of the frontier condition of the stratification $(X_k)_k$, it is now sufficient to prove the case for which there is no $k$ such that $X_i> X_k> X_j$.   Moreover by replacing $x$ and $y$ by their images by some iterates of the dynamics, we can suppose $x$ and $y$ close to $\Lambda_j$.  Then $\tilde W^s_\epsilon(y)$ contains a manifold $P_n(z)$ close o $\tilde W^s_\epsilon(x)= W^s_\epsilon(x)$.
This implies the continuity of $x\in A\mapsto \tilde W^s_\epsilon(x)$.%
%
\end{proof}

\noindent\textbf{Laminar structure on each $X_i$. }

Let $(K_k)_k$ be a filtration adapted to $(\Lambda_k)_k$. 
Let $\ring K_i$ be the interior of $K_i$ for the topology of $K$ induced by $M$.

Let us prove that $X_i^n:=\cup_{x\in f^{-n} (\ring K_i)\cap W^s_i} \tilde W^s_\epsilon(x)$ is laminated by local stable manifolds.

As $f^m(\tilde W_\epsilon^s(x))\subset W_{\lambda^m \epsilon}^s (f^m(x))$ and $\cap_{m\ge 0} f^{m} ( K_i)\cap W^s_i=\Lambda_i$, the set $f^m(X_i^n)$ is included in $W^s_{\epsilon/2}(\Lambda_i)$ for $m$ large. Also $W^s_{\epsilon/2}(\Lambda_i)$ supports a laminar structure by local stable manifolds since $\Lambda_i$ has a local product structure. This structure can be induced on $X_i^n$ if $f^m(X_i^n)$ is an open subset of $W^s_{\epsilon/2}(\Lambda_i)$. For an atlas of $X_i^n$ is then formed by the pull back by $f^m$ of the $W^s_{\epsilon/2}(\Lambda_i)$ charts  with domain included in $f^m(X_i^n)$. Eventually the union of the atlases $(X_i^m)_m$ forms a laminar structure on $X_i$.   

Consequently it remains to show that $f^m(X_i^n)$ is open in $W^s_{\epsilon/2}(\Lambda_i)$. Let $x'\in W^s _{\epsilon/2}(\Lambda_i)$ be close to $x\in f^m(X_i^n)$.

 Let $z\in f^{m-n}(\ring K_i)\cap W^s_i$ be s.t. $x\in f^m(\tilde W^s_{\epsilon}(f^{-m}(z))$. In particular $z$ belongs to $W^s_{\lambda^m\epsilon} (x)\cap K$.
As $m$ is large, $z$ is not far from $x$ which is close to $x'\in W^s_{\epsilon/2}(\Lambda_i)$. 
Also $W^u_\epsilon(z)$ contains a submanifold $C^1$-close to an $\epsilon$-local unstable manifold of $\Lambda_i$ by Lemma \ref{champ de plaque elementaire} applied to the inverse dynamics.

Consequently $W^u_\epsilon(z)$ intersects transversally $W^s_\epsilon(x')$ at a set which contains a point $z'$  close to $z$. By the second condition of $AS$ compact set definition, the point $z'$ belongs to $K$. Thus, the point $z'$ belongs to the open neighborhood $f^{m-n}(\ring K_i)$ of $z\in K$.

 As both $z$ and $z'$ belong to $W^s_i$, Lemma \ref{texnic} implies that $f^m(\tilde W^s_{\epsilon}(f^{-m}(z)))$ and $f^m(\tilde W^s_{\epsilon}(f^{-m}(z')))$ are $C^1$ close. Also $x'$ belongs to $W^s_{\epsilon}(z')\supset f^m(\tilde W^s_{\epsilon}(f^{-m}(z')))$  and is close to $x\in f^m(\tilde W^s_{\epsilon}(f^{-m}(z)))$. Thus $x'$ belongs to $f^m(\tilde W^s_{\epsilon}(f^{-m}(z')))$ when $x'$ is close to $x$. In particular $x'$ belongs to $f^m(X_i^n)$.
\end{proof}

Pulling back this construction by $\pi$, we get on $\hat M$ a stratification of laminations $\hat \Sigma$ whose strata are 
$\hat X_i:= \pi^{-1}(X_i)$ and whose leaves are the connected components of the preimages by $\pi$ of the leaves of $X_i$.

\subsection{Trellis structure on a local stable set of $K$}\label{construction de treillis}
In order to prove the persistence of the stratifications of laminations $\Sigma$ and $\hat \Sigma$, we shall use the existence of trellis structures on $( A, \Sigma)$ and on $(\hat A,\hat \Sigma)$.

A {\it trellis (of laminations)} on a laminar stratified space $(A,\Sigma)$ is a family of
laminations $\mathcal T:=(L_X, \mathcal L_{X})_{X\in \Sigma}$ such that for all strata $X\le Y\in \Sigma$:\begin{itemize}
\item The set $X$ is saturated in $(L_X,\mathcal L_X)$ and $L_X$ is a neighborhood of $X$,
\item each plaque $P$ of $\mathcal L_Y$ included in $L_X$ is $C^1$ foliated by plaques of $\mathcal L_X$,
\item every point, close to the same plaque $P$, is included in a plaque $P'$ of $\mathcal L_Y$ included in $L_X$, and whose foliation by $\mathcal L_X$-plaques is diffeomorphic and $C^1$ close to the one of $P$.\end{itemize} 
 
The lamination $(L_X, \mathcal L_X)$ is \emph{the tubular neighborhood of $X$}.

The tubular neighborhoods of the strata of maximal dimension are equal to their strata since they are open.

In case $\Sigma$ is the canonical stratification induced by an $AS$ compact subset, the $0$-dimensional strata correspond to repulsive periodic orbits.
The tubular neighborhood of these strata is just a neighborhood of them laminated by points, as illustrated by figure \ref{constructionfinal}.d).

To construct the other tubular neighborhoods, here it is much simpler to assume that $M$ is a surface. The existence of tubular nieghborhood is actually an open problem in higher dimension.

In the case $K$ equals $M$ and $f$ is an $AS$ diffeomorphism, we can use the following result of de Melo :
\begin{prop}[\cite{dM}, Thm 2.2 with its Rmk above]
Let $f$ be an $AS$ diffeomorphism of a compact surface. Let $X$ be the stable set of the union of the basic sets with a one dimensional stable direction. 
Then $X$ can be endowed with the 
structure of lamination whose leaves are the stable manifolds.
Moreover there exists a lamination $\mathcal L_X$ on a neighborhood $L_X$ of $X$ such that every leaf of $X$ is a leaf of $\mathcal L_X$ and the leaves of $\mathcal L_X$ that do not intersect 
$X$ form a $C^1$ foliation on $L_X\setminus X$. Furthermore every $\mathcal L_X$-plaque included in $f^{-1}(L_X)$ is sent by $f$ into a plaque of $\mathcal L_X$. 
\end{prop}

This Proposition provides the last tubular neighborhoods of the stratification $\Sigma$: we give the same tubular neighborhood $(L_X,\mathcal L_X)$ for all the basic sets with one dimensional stable direction. 

Let us construct the one dimensional tubular neighborhood of the strata of an $AS$ compact subset of a surface $M$, as illustrated by figures \ref{constructionfinal}.e) and \ref{constructionfinal}.f).
\begin{prop}\label{trellis construit}
Let $f$ be a diffeomorphism of a surface $M$.
Let $K$ be an $AS$ compact subset of $M$. Then for every small $\epsilon>0$, 
the stratification of laminations $\Sigma$ on $A:= \cup_{x\in K} \tilde W^s_\epsilon(x)$ is endowed with a trellis structure $\mathcal T$ such that the tubular neighborhood $(L_k,\mathcal L_k)$ of each $X_k\in \Sigma$ satisfies:
\begin{itemize}
\item Each lamination $\mathcal L_{k}$ is locally $f$-invariant: every plaque of $\mathcal L_{k}$ contained in $f^{-1}(L_{k})$ is sent by $f$ into a plaque of $\mathcal L_{k}$.
\item The tubular neighborhoods are compatible:  for every $x\in L_{k}\cap L_{j}$ with  $j\le k$, every plaque of $\mathcal L_{k}$ containing $x$ is included in a plaque of $\mathcal L_{j}$.
\end{itemize}

\end{prop}

\begin{coro}\label{trellis construit en haut}
Let $\hat X_k$  and $\hat{\mathcal L}_{k}$ be the laminations on $\pi^{-1}(X_k)$ and $\hat L_k:= \pi^{-1}(L_k)$ whose leaves are the connected components of preimages by $\pi$ of the leaves of $ X_k$ and $\mathcal L_{k}$ respectively. Let $\hat A:= \pi^{-1}(A)$ and $\hat \Sigma:= (\hat X_k)_k$. 
The family $\hat {\mathcal T}:= (\hat{L}_{k},\hat{\mathcal L}_{k})_k$ is a trellis structure on the pair $(\hat A,\hat \Sigma)$.  \end{coro}

\begin{proof}[Proof of Proposition \ref{trellis construit}]

As we noticed, it remains only to construct the tubular neighborhood of the one dimensional strata.
Let $J$ be the set of indexes $j$ such that $X_j$ is one dimensional.  We proceed by induction on $k\in J$.
As $(L_0,\mathcal L_0)=X_0$ is a tubular neighborhood of $X_0$ (even if $X_0$ is one dimensional), 
we suppose the tubular neighborhood $(L_j,\mathcal L_j)_{j\le k}$ constructed. This is a trellis structure for the stratification $(X_j)_{j\le k}$ on $A_k:=\cup_{j\le k} X_j$. 

Let us construct $(L_{k+1},\mathcal L_{k+1})$.

By Lemma \ref{HPSP},  for every $\delta$ small enough:
\[\Delta:= cl\Big(W_\delta^u(\Lambda_{k+1})\setminus  f^{-1}\big(W_\delta^u(\Lambda_{k+1})\big)\Big)\]
is a proper fundamental domain for $W^u(\Lambda_{k+1})$. we take $\delta<\epsilon$.

 Let $D$ be a small open neighborhood of $\Delta\cap K$ in $W^u_\epsilon(\Lambda_{k+1})\cap K$. 

Let $\mathcal F_u$ be the family of local unstable manifolds $\{W^u_\epsilon (x);\; x\in D\}$.

By the strong transversality hypothesis, the manifolds of $\mathcal F_u$ are transverse to the leaves of the laminations $(X_j)_{j\le k}$. 

 Let $J':=J\cap \{0,\cdots ,k\}$. As the laminations $(L_j,\mathcal L_j)_{j\in J'}$ are compatible, the union of the atlases forms an atlas of lamination $\tilde {\mathcal L}$ on $\tilde L:= \cup_{j\in J'} L_j$. The union $\tilde X$ of the laminations $(X_j)_{j\in J'}$ is a saturated subset of $(\tilde L, \tilde{\mathcal L})$ since the laminations $(X_j,\mathcal L_j)_{j\in J'}$ are compatible. We endow $\tilde X$ with a laminar structure whose leaves are the leaves of $(X_j)_{j\in J'}$.

The lamination $\tilde X$ is closed in $A_{k+1}$ and so the intersection between $cl(D)$ and $\tilde X$  is compact. Consequently, the transversality between $\tilde X$ and $\mathcal F_u$ is uniform at $D$.

Thus we can shrink each tubular neighborhood $(L_j,\mathcal L_j)_{j\in J'}$, in order that $ (\tilde L,\tilde {\mathcal L})$ is transverse to $\mathcal F^u$ at $D$.

Let $D_{\epsilon}:=\cup_{x\in D} \tilde W_{\epsilon}^s(x)$. It is an open subset of $A_{k+1}:=\cup_{j\le k+1} X_j$ by the following Lemma proved at the end of this section:

\begin{lemm}\label{ouvert}Let $O$ be an open subset of $K$ and let $\epsilon'<\epsilon$. Then the subset $O_{\epsilon'}:= \cup_{x\in D} \tilde W_{\epsilon'}^s(x)$ is open in $A$.
\end{lemm}
%
%
%
%
%
%


We notice that $D_\epsilon\setminus \tilde X$ is open (in $A_{k+1}$) and an open subset of it is foliated by $\tilde {\mathcal L}\setminus \tilde X$. 

We want to extend $(\tilde L\setminus \tilde X, \tilde {\mathcal L}{|\tilde L\setminus \tilde X})$ to a $C^1$ foliation on $\mathcal D_{\epsilon}\setminus \tilde X$ transverse to $\mathcal F_u$ and locally $f$-invariant.  As the leaves of $\tilde {\mathcal L}$ are of dimension one, this would be simple if this foliation would have been of class $C^2$. However it is in general not the case since $f$ is of class $C^1$. The construction is done in three steps:

\begin{itemize}
\item First we suppose $\delta$ small enough in order that $T\tilde {\mathcal L}$ can be continuously extended to  a continuous line field $\chi$ on an open neighborhood of $D_\epsilon$, uniformly transverse to $\mathcal F_u$. Let $\tilde {\chi}$ be a smooth approximation of $\chi$.

\item We smooth (and so lose partially the local invariance) $\tilde{\mathcal L}$ off a neighborhood $\tilde{\dot L}$ of $\tilde X$ by the following Lemma:%

\begin{lemm}[ Lemma 1.4 \cite{dM}]
Let $\mathcal F$ be a one dimensional $C^1$ foliation of an open subset $U$ of a manifold. Let $C$ be a compact subset of $U$ and $\eta>0$. Then for every neighborhood $\hat C$ of $C$ there exists a foliation $\mathcal F'$ which coincides with $\mathcal F$ on a neighborhood of $C$, which is smooth on the complement of $\hat C$ and whose tangent space is everywhere $\eta$-close to the one of $\mathcal F$.    
\end{lemm} 
Then we patch this foliation to the one integrating $\tilde \chi$ in order to form a lamination $\mathcal F$ on $D_\epsilon$ such that:
\begin{itemize} \item $\mathcal F$ restricted to $\tilde{\dot  L}$ is $\tilde {\mathcal L}|\tilde{\dot  L}$,
\item $\mathcal F$ restricted to $D_\epsilon \setminus \tilde L$ is smooth.
\end{itemize}
\item For $\delta'$ sufficiently close to $\delta$, $ B^+:= \cup_{x\in D^+} \tilde W^s_\epsilon(x)$ is disjoint from $B^- := \cup_{x\in D^-} \tilde W^s_\epsilon(x)$, with $D^+:= K\cap W^u_\delta (\Lambda_{j+1})\setminus cl(W^u_{\delta'} (\Lambda_{j+1}))$ and $D^-:= f^{-1}(D^+)$.
 
We define on $B^+\cup B^-\cup \tilde{\dot{L}}$ the lamination $\mathcal F'$ equal to $\mathcal F$ on $B^+\cup \tilde{\dot L}$ and equal to the pullback $f_*\mathcal F$ on $B^-\cup f^{-1}(\tilde{\dot L})$. The lamination $\mathcal F'$ does exist since 
$\tilde{\mathcal F}|\tilde{\dot L}$ is locally $f$-invariant and 
$B^+$ and $B^-$ are disjoint.
 
We complement as above $\mathcal F'$ to a lamination on $D_\epsilon$ which is locally $f-$invariant.
Moreover we can suppose all the approximations done sufficiently small such that $\mathcal F'$ is transverse to the $\epsilon$-local unstable manifolds of $\Lambda_{k+1}$'s points.
\end{itemize}
Let $L_{k+1}$ be the union $X_{k+1}\cup \bigcup_{n\ge 0} f^{-n}(\mathcal F')$ that we laminate by the pull back of the leaves of $\mathcal F'$ and the leaves of $X_{k+1}$. This defines a lamination $\mathcal L_{k+1}$ by the lambda lemma.
\end{proof}
\begin{proof}[Proof of Lemma \ref{ouvert}] Let $x\in O_{\epsilon'}$. We want to prove that $y\in A$ close to $x$ belongs to $O_{\epsilon'}$.
Let $z\in O$ be such that $x$ belongs to $\tilde W^s_{\epsilon'}(z)$. The intersection $\tilde W^s_{\epsilon'}(x)$ with $W^u_{\epsilon}(z)$ is transverse and contains $z$.  By Lemma \ref{texnic}, for $y$ close enough to $x$, the connected submanifold $W^s_{\epsilon'}(y)$ intersected with $W^u_{\epsilon}(z)$ contains a point $z'$ close to $z$. 
The point $z'$ must belong to $K$ by definition of $AS$-compact set, and so to $O$.  Thus $y$ belongs to $O_{\epsilon'}$. 
 \end{proof}

\section{Persistence of trellis structures}
In this section we recall the main Theorem of \cite{Bermem} on persistence of stratifications of laminations.
The reader should look at this paper for more details and examples. 

 Let us first recall a few terminologies. 
\subsection{Plaque-expansiveness}\label{plaque-expansive}
Let $(L,\mathcal L)$ be a compact lamination embedded into a manifold $M$. 
Let $\tilde f$ be a diffeomorphism of $M$ which preserves and leaves invariant $(L,\mathcal L)$. For a positive real number $\eta$, an \emph{$\eta$-pseudo-orbit $(x_n)_{n\in \mathbb Z}$ which respects $\mathcal L$} is a sequence of $L$ such that  for all $n\in \mathbb Z$, $\tilde f(x_n)$ and $x_{n+1}$ 
belong to a same plaque of $\mathcal L$ of diameter less than $\eta$. A diffeomorphism $\tilde f$ is \emph{plaque-expansive} at $(L,\mathcal L)$ if for every small $\eta$, for all $\eta$-pseudo orbits $(x_n)_n$ and $(y_n)_n$ respecting $\mathcal L$ such that $x_n$ and $y_n$ are $\eta$-close for every $n\in \mathbb Z$, then $x_0$ and $y_0$ belong to a same small plaque.
This is the definition used in HPS' Theorem \ref{HPS}.

The persistence Theorem on trellis structure (see below) works also for endomorphisms of manifold ($C^1$ maps that are possibly non invertible).  
Thus instead of considering pseudo-orbits, we shall regard \emph{forward pseudo-orbits} which are pseudo-orbits implemented by $\mathbb N$. Moreover the laminations are not all invariant and not all compact. This leads us to generalize the above concepts in the following form.

Let $(L, \mathcal L)$ be a lamination. Let $f$ be a continuous map from an open subset $V$ of $L$ into $L$.  Let $\eta$ be a continuous, positive function on $L$. An \emph{$\eta$-forward-pseudo-orbit $(x_n)_{n\ge 0}\in V^\mathbb N$ which respects $\mathcal L$}  is a sequence such that 
 for all $n\ge 0$, $f(x_n)$ and $x_{n+1}$ belong to a same plaque of $\mathcal L$ of diameter less than $\eta(x_n)$.

Let $X\subset V$ be a $\mathcal L$-saturated set. We say that $f$ is \emph{$\eta$-forward-plaque-expansive at $X$} relatively to $\mathcal L$ if for every pair of $\eta$-pseudo-orbits $(x_n)_{n\ge 0}\in V^\mathbb N$ and $(y_n)_{n\ge 0}\in V^\mathbb N$ which respect $\mathcal L$ and such that $d(x_n, y_n)\le \eta(x_n)$, then the points  $x_0$ and $y_0$ belong to a same small plaque, which is moreover included in $X$. We say that $f$ is \emph{forward-plaque-expansive at $X$} if $f$ is $\eta$-forward-plaque-expansive at $X$ for every function $\eta$ smaller than a given function. 

\begin{rema}This is a more general definition in the following example. Let $f$ be a D.A. of a torus (see \cite{Sm}). This is an $AS$ diffeomorphism whose basic sets are a repulsive point $R$ and an attractor whose stable manifolds form a lamination $\mathcal L$ which ends at $R$. For every $\eta>0$, nearby $R$, there are many $\eta$-pseudo-orbits which respect $\mathcal L$ and stay $\eta$-close to $R$. They are not all in the same plaque and so $f$ is not $\eta$-forward-plaque-expansive at $\mathcal L$. It is not the case for a suitable positive function, as shown  in the following example.\end{rema}

\begin{exem} \label{example d'expan}
Let $\Sigma$ be the stratification constructed in Proposition \ref{prop2} for an \emph{AS} compact set $K$ of $f$.
Let $X$ be a stratum of $\Sigma$ associated to a basic set $\Lambda$, and let $(L_X,\mathcal L_X)$ be its tubular neighborhood given by Proposition \ref{trellis construit}.
For every $\delta>0$, by taking $L_X$ and a function $\eta$ small enough, every $\eta$-pseudo-orbit eventually lands in the $\delta$-neighborhood of $\Lambda$.
Therefore it is sufficient to prove that for all $\eta$-close $\eta$-pseudo-orbits $(x_n)_n$ and $(y_n)_n$ in $B(\Lambda,\delta)$ which respect $\mathcal L_X$, then $x_0$ and $y_0$ belong to a same small local stable manifold of $\Lambda$. It is indeed the case. For put $x'_n:= f^n(x_0)$ and $y'_n:= f^n(y_0)$. By $\lambda$-contraction along the stable direction of $\Lambda$, 
$x_1'$ is $(\lambda+1)\eta$-close to $x_1$, $x_2'$ is $(\lambda(\lambda+1)+1)\eta$-close to $x_2$, and so on $x'_n$ is $(\sum_{i\le n} \lambda^i)\eta$-close to $x_n$.  
That is why $(x'_n)_n$ is  $\eta/(1-\lambda)$-close to $(x_n)_n$.
Thus we can shadow  $(y_n)_n$, $(x_n)_n$ and $(x'_n)_n$ by the orbit of a same point $z\in \Lambda$. We note that $x_0$ and $y_0$ belong to a same small local stable manifold of $z$.
\end{exem} 
 
\subsection{Statement of the persistence Theorem of stratifications of laminations}

Let $(A,\Sigma)$ be a stratified space endowed with a structure of trellis $\mathcal T$. An \emph{embedding} of $(A,\Sigma,\mathcal T)$ into a manifold $M$ is a homeomorphism from $A$ onto its image in $M$ whose restriction to each tubular neighborhood $(L_X, \mathcal L_X)$ is an immersion of laminations. We recall that the set of embeddings of laminations is endowed with the $C^1$ compact-open topology (see the introduction). 
  We endow the space of immersions $i$ of $(A,\Sigma,\mathcal T)$ into a manifold $M$, with the initial topology of the following inclusion into the product of the spaces of embeddings from $(L_X,\mathcal L_X)$ into $M$ for every ${X\in \Sigma}$:
\[i\mapsto (i{|L_X})_{X\in \Sigma}\]
The initial topology is the coarsest one such that the above map is continuous.  
Given an open subset {${\dot A}\subset A$}, we denote by {$\Sigma{|{\dot A}}$} the stratification of laminations on ${{\dot A}}$ whose strata are the restrictions {$X{|X\cap {\dot A}}$} of  the strata $X\in \Sigma$ to {$X\cap {\dot A}$}. Similarly the trellis structure {$\mathcal T{|{\dot A}}$} is made by restricting each tubular neighborhood $(L_X,\mathcal L_X)$ to $L_X\cap \dot A$.

\begin{theo}[Cor. 2.2.9 \cite{Bermem}] \label{main} Let $(A,\Sigma)$ be a stratified space supporting a trellis structure $\mathcal T$. Let $i$ be an embedding of $(A,\Sigma, \mathcal T)$ into a manifold $M$.
Let $f$ be a $C^1$ map from $M$ into itself,  which sends the closure of $A$  into $A$ such that:

\begin{enumerate}[i.]
\item $f$ preserves and normally expands each stratum $X$ of $\Sigma$,
\item $f$ sends each plaque of $\mathcal L_X$ included in $f^{-1}(L_X)$ into a plaque of $\mathcal L_X$,
\item $f{|f^{-1}(L_X)}\cap L_X$ is forward plaque-expansive at $X$ relatively to $\mathcal L_X$, for every $X\in \Sigma$.
\end{enumerate}
Let $\dot A$ be an open subset of $A$ whose closure is compact, included in $A$ and sent by $f$ into $\dot A$.

\indent  Then for $f'$ $C^1$ close to $f$,  
  there exist an embedding $i'$ of $({\dot A},\Sigma{|{\dot A}}, \mathcal T{|{\dot A}})$ into $M$, close to $i|\dot A$ and a continuous map $f'^*$ $C^0$-close to $f^*:= f{|{\dot A}}$ such that the following diagram commutes:
\[\begin{array}{rcccl}
&& f'&&\\
 &   M    &\rightarrow &M'&\\
i'&\uparrow& & \uparrow&i'\\
 &{\dot A}       &\rightarrow&{\dot A}& \\
 &&f'^*&&\end{array}\]

 Moreover, there exists a family of neighborhoods $({ V}_X)_{X\in \Sigma}$  in ${\dot A}$ of $({\dot A}\cap X)_{X\in \Sigma}$ respectively  such that
for every $f'$ $C^1$ close to $f$, the map $f'^*$ sends each plaque of $\mathcal L_X$ included in ${ V}_X$ into the same  leaf of $\mathcal L_X$ as $f^*$.
 \end{theo}

Let us show how to apply this Theorem to prove the main Theorem. 

Let $(A,\Sigma)$ and $\mathcal T$ be the stratification of laminations and its trellis structure given by Proposition \ref{trellis construit} for the $AS$ compact set $K$ for $f$. Let $(\hat A,\hat \Sigma)$ and $\hat {\mathcal T}$ be those given by Corollary \ref{trellis construit} for $\hat f$ as in the main Theorem. 

Let us apply Theorem \ref{main} to the latter stratified space and trellis, with the dynamics $\hat f$.

%

The preservation of $\hat \Sigma$ is clear. Property $ii$ follows immediately from the second item of Proposition \ref{trellis construit}. The normal expansion of $\hat X_i:= \pi^{-1}(X_i)$ follows from the normal hyperbolicity of $\hat \Lambda_i:= \pi^{-1}(\Lambda_i)$ and from the belonging of the leaves of $X_i$ to stable manifolds of $\Lambda_i$. 

Let us show Property $iii$ for each stratum $\hat X_i\in \hat \Sigma$. 

 Let $(\hat x_n)_n$ and $(\hat y_n)_n$ be two close $\eta$-pseudo-orbits of $\hat f^{-1}(\hat L_i)$ which respect the plaques of $\hat {\mathcal L}_{i}$. Let $(x_n)_n$ and $(y_n)_n$ be the images by $\pi$ of these pseudo-orbits. By example \ref{example d'expan}, for $\eta$ small enough, $x_0$ and $y_0$ belong to a same small plaque of $X_i$.  We can suppose that $\eta$ is small enough so that this plaque is included in a trivialization of $\pi:\; \hat M\rightarrow M$. As $\hat x_0$ and $\hat y_0$ are $\eta$-close and belong to the same preimage of a small plaque of a trivialization, they  belong to a same small plaque of $\hat {\mathcal L}_{i}$. 

Take $\epsilon'<\epsilon$ and put:
\[{{\dot A}}:= \cup_{x\in K} \tilde W^s_{\epsilon'}(x).\]
By Lemma \ref{ouvert}, the subset ${\dot A}$ is open in $A$. By Fact \ref{factpedago}, for $x\in K$, the closure of each $\tilde W^s_{\epsilon'}(x)$ is included in $\tilde W^s_{\epsilon}(x)$ and so the closure of $\dot A$ is included in $A$. Furthermore, as $f$ sends each $\tilde W^s_{\epsilon'}(x)$ into $\tilde W^s_{\lambda \epsilon'}(f(x))$, it sends $cl(\dot A)$ into $\dot A$.

Consequently we can apply Theorem \ref{main} to $(\hat A,\hat \Sigma)$, $\hat{\mathcal T}$ and $\hat {{\dot A}}:=\pi^{-1}(\dot A)$.

 For $\hat f'$ close to $\hat f$, we recall that $i'$ denotes the embedding of {$(\hat {{\dot A}},\hat \Sigma{|\hat{{\dot A}}})$} and $(\hat V_{i}):=(V_{\hat X_i})_i$ the family of neighborhoods given by the above Theorem. 
By shrinking $\hat V_i$ we can suppose that $\hat V_i:=\pi^{-1}(V_i)$ with $V_i:=\pi(\hat V_i)$.
%
%

\section{Proof of main Theorem \ref{Main result}}
\subsection{Tubular neighborhood of $\mathcal L$}
Let $y\in \hat M\mapsto F_y$ be a smooth section of the Grassmannian of $T\hat M$ $C^0$-close to the orthogonal of the tangent space to the fibers of
$\pi$. Hence we can suppose that for $y\in \hat M$:
\[F_y\oplus \ker T_y\pi =T_y \hat M\]

We endow  $F:=\sqcup_{y\in \hat M} F_y$ with the canonical vector subbundle structure of $T\hat M$.
Let $\alpha>0$ be a small enough such that for every $x\in L$,

\[\{(y,u)\in F\; :\; \|u\|< \alpha,\; y\in \mathcal L_x\}\]
is embedded by the exponential map $\exp$ associated to the Riemannian metric of $\hat M$.

 \[\mathrm{Let }\; Exp\; :\; F\rightarrow \hat M\]
\[(y,u)\mapsto \exp_y\left(
\frac{\alpha\cdot u}{\sqrt{1+\|u\|^2}}
\right).\]
The factor $\alpha/\sqrt{1+\|u\|^2}$ makes, in particular, the restriction of $Exp$ to each fiber $ F_y$ a diffeomorphism onto its image denoted by $\mathcal F_y$, for every $y\in \hat M$.

We notice that the restriction of $Exp$ to the zero section of $F$ is equal to the canonical inclusion  of $\mathcal L$ into $\hat M$.

Let $G$ be the set formed by the submanifolds of $\hat M$ $C^1$ diffeomorphic to a fiber of $\pi$.
We endow $G$ with the $C^1$ uniform topology. The submanifolds family $(\mathcal F_{y})_{y\in L}$ is useful for the following: 

\begin{claim} \label{Pour holonomy}There exists an open neighborhood $V_G$ of $\{(x,\mathcal L_x); x \in K\}$ in $M × G$, such that for 
every $(x, N) \in  V_G$, every $y\in \mathcal L_x$, the submanifolds $\mathcal F_y$ and  $N$ have a transverse intersection which consists of a unique point $I(y, N)$.

Moreover the map $I$ is continuous and its differential with respect to the first variable along $\mathcal L_X$ exists, is injective and depends continuously on 
$(x, N)\in V_G$.\end{claim}
\begin{proof} For every $x\in K$, every $N\in G$ at a neighborhood of $\mathcal L_x$ is a transverse section of the $C^1$-foliation formed by the leaves $(\mathcal F_y)_{y\in \mathcal L_x}$. The holonomy from the section $\mathcal L_x$ to the section $N$ defines a $C^1$ maps $I(\cdot, N)|\mathcal L_x:\; y\in \mathcal L_x \mapsto I(y,N)$. 

For $x'$ close to $x$, the manifold $\mathcal L_{x'}$ is $C^1$-close to $\mathcal L_x$  and the foliation $(\mathcal F_y)_{y\in \mathcal L_{x'}}$ is $C^1$ close to $(\mathcal F_y)_{y\in \mathcal L_x}$. Thus the map  $I(\cdot, N)|\mathcal L_x$ depends continuously on $x'$ nearby $x$ and $N\in G$ nearby $\mathcal L_x$ for the $C^1$-topology.    

We cover the set $\{(x,\mathcal L):\; x\in K\}$ by such neighborhoods in $K\times G$; the union of this covering is the domain $V_G$ of the requested  $I$.
\end{proof}

For $x\in K$ the subset of manifold $N$ such that $(x,N)$ is in $G$ is called \emph{the fiber} of $x$ in $G$.  

We now proceed to the proof of Theorem \ref{Main result} by distinguish two cases: Whether or not $(\Lambda_i)$ contains \emph{one dimensional pieces}, that is pieces whose stable dimension is one.
 
Along the proof $\hat f'$ is supposed closer and closer to $\hat f$, and this even after that the numbers  $\eta$, $\delta$, $\theta$, $\sigma$ will be implemented.
 
\subsection{First step for the case without one dimensional pieces}

Let us suppose there is none one dimensional pieces in $(\Lambda_i)_i$. We suppose that $K$ contains sources $(\Lambda_i)_{i> k}$ and sinks $(\Lambda_i)_{i\le k}$ otherwise $K=\sqcup_i \Lambda_i$ and Theorem \ref{Main result} is an immediate consequence of Theorem \ref{HPS}. We recall the existence of a filtration $(K_i)_i$ of $K$ adapted to $(\Lambda_i)_i$.

We notice that the leaves of the laminations $(\hat {\mathcal L_i})_{i>k}$ given by Corollary \ref{trellis construit en haut} 
constructed in section \ref{construction de treillis} have the same dimension as $\mathcal L$. 

Theorem \ref{main} provides  a family of neighborhoods $(\hat V_i)_i$ of the strata of $\hat \Sigma|\dot A$, and for $\hat f'$ close to $f$,  an embedding $i'$ of $\hat \Sigma|\hat{\dot A}$ preserved by $\hat f'$. 
 
For $M$ large enough, the compact set   
$\hat W:=\pi^{-1}(W)$ is included in the neighborhood $\cup_{i>k} \hat V_i$ of $(\hat \Lambda_i)_{i> k}$, with  $W=cl(K\setminus f^{-M}(K_k))$.
Also we note that $W$ is included in the basin of repulsion of $(\Lambda_i)_{i\le  k}$.

We denote by $h:\; y\in \hat W\mapsto I(y, i'(\mathcal L_x))$, with $x:= \pi(x)$.

\begin{Property}\label{premiere abcd}
\begin{itemize}
\item[(a)] $h$ is an embedding from the lamination $\mathcal L{|\hat W}$ into $M$, which is $C^1$ close to the canonical inclusion,
\item[(b)] For every $x\in f^{-1}(W)$,  $\hat f'$ sends $h(\mathcal L_x)$ onto $h(\mathcal L_{f(x)})$:
\[ \hat f'(\mathcal L'_x)=\mathcal L'_{ f(x)}\quad \mathrm{with}\quad \mathcal L'_x:= h(\mathcal L_x).\]
\item[(c)] For every $y\in \hat W$,  $h(y)$ belongs to $\mathcal F_y$.
\item[(d)] For every $n\ge 0$, every $x,x'\in W$ such that $\hat f'^n (\mathcal L_x')$ intersects $\mathcal L_{x'}'$ satisfy that 
$x'= f^n(x)$.
\end{itemize}
\end{Property}
\begin{proof} Property {\it (c)} is obvious. Property {\it(a) } follows from Claim \ref{Pour holonomy}. Property \textit{(b)} holds since $\hat W$ is included in the union $\cup_{i>k} \hat V_i$ of neighborhoods adapted to $\hat f'$.

Property {\it (d)} for $n\le 1$ follows from the fact that $W\subset \cup_{i>k} \hat V_i$ and that the trellis structure is embedded. For $n\ge 2$, we notice that $f^2(W)\setminus f(W)$ is a compact set disjoint from $W$ included in the attraction basin of $(\Lambda_i)_{i\le k}$. Thus the image $\hat f'^n( \mathcal L_x)$ is far from $\mathcal L'_{x'}$ for every  $n\ge 2$ and $x,x'\in W$.
\end{proof}

We extend $h$ to $K$ in section \ref{second step}.

\subsection{First step for the case with one dimensional pieces}
Let us now suppose there are one dimensional pieces in $(\Lambda_i)_i$. 

We suppose the basic pieces $(\Lambda_i)_{i=0}^N$ indexed such that 
 attracting cycle have indexes at most $k$, the one dimensional pieces have indexes in $[k+1, l]$, the repulsive cycle have indexes greater than $l$, where  $-1\le k< l\le N$.

We recall that the tubular neighborhoods $(\hat L_i, \hat {\mathcal L_i})_{i=k+1}^l$ given by Corollary \ref{trellis construit en haut} are compatible and so form a lamination $\mathcal L^s$ on the union $L^s=\cup_{i=k+1}^l \hat L_i$.

For every $\hat f'$ nearby $\hat f$, let $i^s$ be the restriction to $L^s$ of the embedding $i'$ given by Theorem \ref{main} applied to $\hat f$ and the structures given by Corollary \ref{trellis construit en haut}.
 We remark that $i^s$ is an embedding of $\mathcal L^s$ close to the canonical inclusion.
 We denote by $V^s\subset L^s$ the union of the adapted neighborhoods $\cup_{i=k+1}^l \hat V_i$. We remark that $V^s$ is a neighborhood of $\pi^{-1}(\cup_{i=k+1}^l W^s_i)$. 
 
 Furthermore every $y\in V^s$ has its image $i^s(y)$ sent by $\hat f'$ into the image by $i^s$ of a small $\mathcal L^s$-plaque of $\hat f(y)$.
 
 We do exactly the same construction for the inverse dynamics. We denote by $(\mathcal L^u, L^u)$, $i^u$ and $V^u$ the lamination, embedding and neighborhood obtained for $\hat f'^{-1}$ close to $\hat f^{-1}$.
 
\begin{Property} We can fix  $M$ large enough such that $V^s\cap \hat f^{-1}(V^u)$ is a neighborhood (in $L$) of $\hat W$ with:
\[W:= cl\big(f^M(K_l)\setminus f^{-M}(K_k)\big),\quad \mathrm{and}\quad \hat W:=\pi^{-1}(W)\]
\end{Property}
\begin{proof} By Property \ref{propfiltration}, we have:
\[\bigcap_{M>0} cl(f^M(K_l)\setminus f^{-M}(K_k))=\bigcap_{M>0} f^M(K_l)\cap \bigcap_{M>0} f^{-M} (K_k^c)\]
\[=\bigcup_{i\le l} W^u_i\cap \bigcup_{j> k}W^s_j=\bigcup_{k< i\le l} W^u_i\cap \bigcup_{k< j\le l}
W^s_j.\]
As $V^s$ and $V^u$ are neighborhoods of the left side and righ side of the latter intersection, it is a neighborhood of $\bigcap_{M>0} cl(F^M(K_l)\setminus f^{-M}(K_k))$ and so we can conclude by compactness.\end{proof}

For $\eta>0$ and $x\in W\cup f(W) \cup f^{-1}(W)$, we denote by $\mathcal L^s_\eta(x)$ the union of $\mathcal L^s$-plaques of diameter less than $\epsilon$ which intersect $\mathcal L_x$. Similarly, we define $\mathcal L^u_\eta(x)$

We fix $\eta<\delta$ small enough so that the following condition holds for $\hat f'$ close enough to $\hat f$ and $x\in W\cup f(W) \cup f^{-1}(W)$:

\begin{itemize}
  \item[\it (i)] the transverse intersection $\mathcal L'_x:= i^s(\mathcal L^s_\eta(x))\pitchfork i^u(\mathcal L^s_\eta(x))$ belongs to the fiber of $x$ in $V_G$,   
 \item[\it (ii)] the transverse intersection $\mathcal L'_x:= i^s(\mathcal L^s_\delta(x))\pitchfork i^u(\mathcal L^s_\delta(x))$ belongs also to the fiber of $x$ in $V_G$,
 \item[\it (iii)] $\hat f'$ sends $i^s(\mathcal L^s_\eta(x))$ into $i^s(\mathcal L^s_\delta(f(x)))$,
 \item[\it (iv)] $\hat f'$ sends $i^u(\mathcal L^u_\eta(x))$ into $i^u(\mathcal L^u_\delta(f(x)))$.
 \end{itemize} 
 
We put:
\[h:= y\in \hat W\mapsto I\Big(y, i^s\big(\mathcal L^s_\eta(x)\big)\pitchfork i^u\big(\mathcal L^u_\eta(x)\big)\Big),\quad \mathrm{ with}\; x:=\pi(y).\]

\begin{Property}
\begin{itemize}
\item[(a)] The map $h$ is an embedding from the lamination $\mathcal L{|\hat W}$ into $M$, which is $C^1$ close to the canonical inclusion.
\item[(b)] For every $x\in f^{-1}(W)\cap W$,  $\hat f'$ sends $h(\mathcal L_x)$ onto $h(\mathcal L_{f(x)})$:
\[\mathcal L'_x= h(\mathcal L_x)\quad \mathrm{and}\quad h(\mathcal L_{ f(x)})= \hat f'\circ h(\mathcal L_x).\]
\item[(c)] For every $y\in \hat W$,  $h(y)$ belongs to $\mathcal F_y$.
\item[(d)] For every $x,x'\in W$, $n\in \mathbb Z$, if $\hat f'^n(\mathcal L_{x}')$ intersects $\mathcal L_{x'}'$ then $x'= f^n(x)$.
\end{itemize}
\end{Property}
\begin{proof}
Property \textit{(c)} is obvious. Property {\it (d)} is proved as in Property \ref{premiere abcd}.

Let us prove Property \textit{(a)}. By Claim \ref{Pour holonomy}, the map  $h$ is an immersion of the lamination $\mathcal L{|W}$ into $\hat M$ close to the canonical inclusion, for $\hat f'$ close to $\hat f$. 
Let us show that $h$ is injective. As $h$ is close to the canonical inclusion, if two different points of $W$ have the same image by  $h$, they belong to different manifolds $\mathcal L_x$ and $\mathcal L_{x'}$.  However  $i^s$ and $i^u$ are embeddings of laminations, and so the manifolds $\mathcal L_x$ and $\mathcal L_{x'}$ belong to the same intersection of ${\mathcal L}^u_\delta(x)$ with ${\mathcal L}^s_{\delta}(x)$, equal to the single manifold $\mathcal L_x$. This is a contradiction. 

By compactness of $W$, the map $h$ is an embedding.

Let us prove Property \textit{(b)}. The continuous map $h$ sends $\mathcal L_x$ onto a compact subset of $\mathcal L_x'$. As $h$ is an immersion, $h(\mathcal L_x)$ is also an open subset of $\mathcal L_x'$. Since $h$ is close to the canonical inclusion, by connectedness, $\mathcal L'_x=h(\mathcal L_x)$.
 
By {\it (i)- (iv)}, $f'$ sends each $\mathcal L'_x$ into $\mathcal L'_{f(x)}$ for $x\in W\cap f^{-1}(W)$. For the same reason the first equality holds.

\end{proof}

\begin{rema}\label{therema} By {\it (c)}, for every $y\in \hat W\cap \hat f(\hat W)$, $\hat h(y)=I(y,h(\mathcal L_x))$, with $x=\pi(y)$, since $h(\mathcal L_x)$ is close to $\mathcal L_x$ by {\it (a)}. By {\it (b)}, $h(\mathcal L_x)$ equals $\hat f'\circ h(\mathcal L_{f^{-1}(x)})$. Consequently $h(y)=I\big(y,\hat f'\circ h(\mathcal L_{f^{-1}(x)})\big)$. 
\end{rema}

\subsection{Second step of the construction}\label{second step}
Let us now extend $h$ to $f^{-M}(K_k)\cup W$, such that Properties {\it (a)-(b)-(c)-(d)} are still satisfied. 

We recall that $(\Lambda_i)_{i\le k}$ are attracting cycles and so their preimages $(\hat \Lambda_i)_{i\le k}$ by $\pi$ are normally contracting submanifolds.
By  Theorem \ref{HPS}, the submanifolds  $(\hat \Lambda_i)_{i\le k}$ persist to submanifolds $(\hat \Lambda_i')_{i\le k}$ preserved by
$\hat f'$ close to $\hat f$.

By normal contraction of $\hat \Lambda_i$, 
for every $\theta>0$ small enough and then $\sigma>0$ small, the cone field $\chi$ on $\hat M$ whose cone at $y\in \hat M$  is $\chi(y):= \{v\in T_yM\setminus\{0\}: \angle(v, Ker\, T_y\pi)<\theta\}$ satisfies the following properties for $\hat f'$ close enough to $\hat f$:

\begin{enumerate}
\item The $\sigma$-neighborhood $\hat V$ of $\sqcup_{i\le k} \hat \Lambda_i$ is sent into itself by $\hat f'$. 
\item  For every $y\in \hat V$:
\[cl\Big(T\hat f'\big(\chi(y)\big)\Big)\subset \chi\big(\hat f'(y)\big)\cup \{0\}.\]
\item The intersection between $\chi(y)$ and the tangent space of any manifold $\mathcal F_z=Exp(F_z)$ is empty, for every $y\in \hat V$.
\item For every $x\in K$, every submanifold $N$ belongs to the $G$ fiber of $x$, if the tangent space of $N$ is in $cl(\chi)$ and if the manifolds $N$ and $\mathcal L_x$ are in the same component of $\hat V$. 
\end{enumerate}

By Property \ref{propfiltration}, there exists  $M'>0$ such that $\pi^{-1}(f^{M'-1}(K_{k}))$ is included in $\hat V$.

We suppose that $\hat f'$ is sufficiently close to $\hat f$ in order that $(f^n(x),\hat f'^n(\mathcal L'_x))$ belongs to $V_G$ for every $x\in W$ and $n\le  M+M'$.

By remark \ref{therema}, for $\hat f'$ close enough to $\hat f$, we can extend $h$ to $\bigcup_{n=0}^{M+M'} \hat f^n(\hat W)$ by:
\[y\in \bigcup_{n=0}^{ M+M'} \hat f^n(\hat W)\mapsto I\big(y, \hat f'^{n}(\mathcal L'_{f^{-n}(x)})\big),\quad \mathrm{if}\quad y\in \hat f^n(\hat W).\]

Such an extension is still denoted by $h$. By remark \ref{therema}, $h$ is smooth. By {\it (d)}, $h$ is an embedding.

We remark that every $\hat f'$ close enough to $\hat f$ satisfies for every $x\in f^{M'-1}(K_k\setminus f^{-1}(K_k))$ that 
$\mathcal L'_x$ is included in $\hat V$ and its tangent space is in $cl(\chi)$.
By $\hat f'$-stability of these sets, for every $n\ge0$, $\hat f'^n (\mathcal L'_x)$ is included in $\hat V$ and its tangent space is in $cl(\chi)$. Consequently $\big(f^n(x),\hat f'^n (\mathcal L'_x)\big)$ belongs to $V_G$. 

Thus we can define the extension on $\hat U:= \cup_{n\ge 0} \hat f^n(\hat W)\cup \bigsqcup_i \hat\Lambda_i$:
\[y\in \hat U\mapsto \left\{\begin{array}{cl}
I(y,\hat f'^n (\mathcal L'_x))& \mathrm{if}\; x:=f^{-n}\circ \pi(y)\in W,\\
I(y,\hat \Lambda_i')& \mathrm{if}\; y\in \hat \Lambda_i'\; \mathrm{and} \; i\le k .\end{array}
\right.\] 
that we denote still by $h$.

 We still denote by $\mathcal L'_x$ the manifold $h(\mathcal L_X)$ for every $x\in U:= \pi(\hat U)$. We remark that $h$ satisfies :
\begin{itemize}
\item[\it (b)] For every $x\in U$,  $\hat f'$ sends $h(\mathcal L_x)$ onto $h(\mathcal L_{f(x)})$:
\[ \hat f'(\mathcal L'_x)=\mathcal L'_{ f(x)}\quad \mathrm{with}\quad \mathcal L'_x:= h(\mathcal L_x).\]
\item[\it (c)] For every $y\in \hat U$,  $h(y)$ belongs to $\mathcal F_y$.
\item[\it (d)]  For every $x,x'\in U$, $n\in \mathbb Z$, if $\hat f'^n(\mathcal L_{x}')$ intersects $\mathcal L_{x'}'$ then $x'= f^n(x)$.
\end{itemize}

\paragraph{Proof that $h$ is an immersion}

It remains only to prove the continuity of $h$ and the continuity of its differential with respect to  $T\mathcal L$. As $ f$ sends $K_k$ into its interior and since $\cap_{n\ge 0} f^n(K_k)$ is equal to $\sqcup_{i\le k} \Lambda_i$, we only need to show this on  $\sqcup_{i\le k} \hat \Lambda_i$.

Let us begin by proving of the continuity of $h$. 
Let $(y_n)_n\in K_k^\mathbb N$ be converging to $y\in \sqcup_{i\le  k} \hat \Lambda_i$.
Let $x_n:= \pi(y_n)$ and $x:=\pi(y)$.

By Arzelà-Ascoli Theorem, there exists an accumulation point  $\mathcal L''$ of  $(\mathcal L'_{x_n})_n$ which is a Lipschitz submanifold close to $\mathcal L_x$.  

By uniqueness of the normally hyperbolic manifolds (See \cite{HPS}), the manifolds $\mathcal L_x$ and $\mathcal L''$ are equal. Thus any accumulation point of $(h(y_n))_n=\big(I(y_n,  \mathcal L'_{x_n})\big)_n$ converges to $h(y)=I(y,\mathcal L'_x)$. This shows the continuity of $h$.

Let us show by the sake of contradiction  the continuity of the  derivative of $h$ with respect to $T\mathcal L$ at $\sqcup_{i\le k} \hat \Lambda_i$.
Let $(y_n)_n\in \hat U^\mathbb N$ which converges to $y\in \sqcup_{i\le k} \hat \Lambda_i$, such that $(T h(y_n))_n$ does not converge to $Th(y)$. As  $Th(T_{y_n}\mathcal L)$ is included in $cl(\chi)$ for every  $n\ge 0$,
we can suppose that $(Th_i(T_{y_n}\mathcal L))_n$ converges to a  $d$-plane $P'\subset cl(\chi)$ different of $P:=Th(T_y\mathcal L)\subset cl( \chi)$. By Property {\it (b)}, $T\hat f'^{-k} (P)$ and $T\hat f'^{-k} (P')$ are included in the closure of  $\chi(\hat f'^{-k}\circ h(y))$ for every  $k$ large. By projective hyperbolicity of the cone field $\chi$, we get a contradiction.

\paragraph{Proof that $h$ is an embedding}  
It is sufficient to show that $h$ is injective since $\hat U$ is compact. Also by Property $(c)$, it is sufficient to prove that two different leaves of $\mathcal L{|\hat f^M(\hat K_l)}$ have disjoint images by  $h$.

 The injectivity of $h$ on $\cup_{n= 0}^m \hat f^n(\hat W)$ follows from an easy induction on $m$ by using {\it (d)}. This implies the injectivity of $h$ on $\cup_{n\ge 0} \hat f^n(\hat W)$.
 
 The injectivity of $h$ on $\sqcup_{i\le k} \hat \Lambda_i$ is obvious.

Let $y\in \sqcup_{i\le k} \hat \Lambda_i$ and $y'\in \hat U$ be sent to the same image by $h$. Let $x:= \pi(y)$ and $x':= \pi(y')$. 

The point $y'$ cannot belong to a certain $\hat f^n(\hat W)$ since this would imply that $\mathcal L_{f^{-n} (x')}=\hat f'^{-n}(\mathcal L_x) $ intersects $\mathcal L_{f^{-n} (x)}=\hat f'^{-n}(\mathcal L_{x'})$, which is impossible since $\hat W$ and $\sqcup_{i\le k} \hat \Lambda_i$ are disjoint compact set.

Thus the point $y'$ belongs to the manifold $\sqcup_{i\le k} \hat \Lambda_i$, and so is equal to $y$. This proves the injectivity of $h$.

\paragraph{Conclusion}
In the case where $(\Lambda_i)_i$ does not contain one dimensional pieces, then we are done since $L=\hat U$.

Otherwise we proceed following the same way, but by inversing the dynamics, to extend $h$ to $K$.

\bibliographystyle{alpha}
\bibliography{references}

\end{document}